\documentclass[a4paper]{amsart}
\pdfoutput=1
\usepackage[utf8]{inputenc}
\usepackage[T1]{fontenc}
\usepackage{lmodern}
\usepackage{amssymb}
\usepackage{mathtools,enumitem,dsfont}
\usepackage{stackengine}
\usepackage{microtype}

\usepackage{tikz,tikz-cd}
\usetikzlibrary{arrows}
\tikzset{cd/.style=matrix of math nodes} 
\tikzset{cdar/.style=->,auto}

\tikzset{dar/.style={double,double equal sign distance,-implies}}

\usepackage[colorlinks=true,linkcolor=blue, citecolor=blue,
urlcolor=blue, raiselinks=true,
pdftitle={Existence of groupoid models for diagrams of groupoid correspondences},
pdfauthor={Joanna Ko and Ralf Meyer},
pdfsubject={Mathematics}
]{hyperref}

\setlist[enumerate,1]{label=\textup{(\arabic*)}}
\setlist[enumerate,2]{label=\textup{(\alph*)}}
\usepackage{xcolor}

\usepackage[lite]{amsrefs}

\usepackage{amsthm}
\newcommand{\longref}[2]{\hyperref[#2]{#1~\textup{\ref*{#2}}}}

\numberwithin{equation}{section}
\theoremstyle{plain}
\newtheorem{theorem}[equation]{Theorem}
\newtheorem{lemma}[equation]{Lemma}
\newtheorem{proposition}[equation]{Proposition}

\newtheorem{corollary}[equation]{Corollary}

\theoremstyle{definition}
\newtheorem{definition}[equation]{Definition}

\theoremstyle{remark}
\newtheorem{remark}[equation]{Remark}
\newtheorem{example}[equation]{Example}

\newcommand*{\nb}{\nobreakdash}

\newcommand*{\U}{\mathcal U}
\newcommand{\V}{\mathcal{V}}

\newcommand*{\lc}{\mathrm{lc}}

\newcommand*{\Qu}{\mathsf{p}}
\newcommand*{\s}{s} 
\newcommand*{\rg}{r}

\newcommand*{\Cat}[1][C]{\mathcal #1}
\newcommand*{\Bisp}[1][X]{\mathcal #1}
\newcommand{\Gr}[1][G]{\mathcal #1}

\newcommand{\prto}{\twoheadrightarrow}

\newcommand*{\prop}{\mathrm{prop}}
\newcommand*{\Grcat}{\mathfrak{Gr}}

\newcommand*{\Cont}{\mathrm C}
\newcommand*{\Contb}{\mathrm{C_b}}

\newcommand*{\Cst}{\texorpdfstring{\textup C^*}{C*}}
\newcommand*{\Star}{$^*$\nobreakdash-}

\newcommand*{\defeq}{\mathrel{\vcentcolon=}}
\newcommand*{\congto}{\xrightarrow\sim}

\newcommand*{\id}{\mathrm{id}}

\newcommand{\C}{\mathbb{C}}
\newcommand{\N}{\mathbb{N}}

\newcommand{\Bis}{\mathcal{S}}
\newcommand{\IS}{\mathcal{I}}

\newcommand{\Grcomp}{\circ}

\newcommand{\gps}[1]{{#1}\text{-}\mathsf{Space}}
\newcommand{\pgps}[1]{{#1}\text{-}\mathsf{Space}_\prop}

\newcommand{\rscc}[1]{\beta_{#1}}
\newcommand{\rsc}{\beta_{\Gr^0}}

\DeclarePairedDelimiter{\norm}{\lVert}{\rVert}
\DeclarePairedDelimiterX{\braket}[2]{\langle}{\rangle}{#1\,\delimsize\vert\,\mathopen{}#2}
\DeclarePairedDelimiterX{\setgiven}[2]{\{}{\}}{#1\,{:}\,\mathopen{}#2}

\begin{document}
\title[Existence of groupoid models]{Existence of groupoid models for diagrams of groupoid correspondences}
\author{Joanna Ko}
\email{joanna.ko.maths@gmail.com}
\author{Ralf Meyer}
\email{rmeyer2@uni-goettingen.de}
\address{Mathematisches Institut\\
  Universität Göttingen\\
  Bunsenstraße 3--5\\
  37073 Göttingen\\
  Germany}

\keywords{étale groupoid; groupoid correspondence; bicategory;
  terminal object; relative Stone--Čech compactification}

\begin{abstract}
  This article continues the study of diagrams in the bicategory of
  étale groupoid correspondences.  We prove that any such diagram
  has a groupoid model and that the groupoid model is a locally
  compact étale groupoid if the diagram is locally compact and
  proper.  A key tool for this is the relative Stone--Čech
  compactification for spaces over a locally compact Hausdorff
  space.
\end{abstract}
\subjclass[2020]{18A30; 18N10; 18B40}

\maketitle


\section{Introduction}
\label{sec:intro}

Many interesting \(\Cst\)\nb-algebras may be realised as
\(\Cst\)\nb-algebras of étale, locally compact groupoids.  Examples
are the \(\Cst\)\nb-algebras associated to group actions on spaces,
(higher-rank) graphs, and self-similar groups.  These examples of
\(\Cst\)\nb-algebras are defined by some combinatorial or dynamical
data.  This data is interpreted in
\cites{Antunes-Ko-Meyer:Groupoid_correspondences,
  Meyer:Diagrams_models} as a diagram in a certain bicategory, whose
objects are étale groupoids and whose arrows are called groupoid
correspondences.  A groupoid correspondence is a space with
commuting actions of the two groupoids, subject to some conditions.
In favourable cases, the \(\Cst\)\nb-algebra associated to such a
diagram is a groupoid \(\Cst\)\nb-algebra of a certain étale
groupoid built from the diagram.  A candidate for this groupoid is
proposed in~\cite{Meyer:Diagrams_models}, where it is called the
groupoid model of the diagram.

Here we prove two important results about groupoid models.  First,
any diagram of groupoid correspondences has a groupoid model.
Secondly, the groupoid model is a locally compact groupoid provided
the diagram is proper and consists of locally compact groupoid
correspondences.  The latter is crucial because the groupoid
\(\Cst\)\nb-algebra of an étale groupoid is only defined if it is
locally compact.

By the results in~\cite{Meyer:Diagrams_models}, the groupoid model
exists if and only if the category of actions of the diagram on
spaces defined in~\cite{Meyer:Diagrams_models} has a terminal
object, and then it is unique up to isomorphism.  To show that such
a terminal diagram action exists, we prove that the category of
actions is cocomplete and has a coseparating set of objects; this
criterion is also used to prove the Special Adjoint Functor Theorem.

Proving that the groupoid model is locally compact is more
challenging.  The key ingredient here is the relative
Stone--\v{C}ech compactification.  This is defined for a space~\(Y\)
with a continuous map to a locally compact Hausdorff ``base
space''~\(B\), and produces another space over~\(B\) that is proper
in the sense that the map to~\(B\) is proper and its underlying
space is Hausdorff.  If~\(B\) is a point, then the relative
Stone--\v{C}ech compactification becomes the usual Stone--\v{C}ech
compactification.  An action of a diagram on a space~\(Y\) contains
a map \(Y\to \Gr^0\) for a certain space~\(\Gr^0\), which is locally
compact and Hausdorff if and only if the diagram is locally compact.
If the diagram is proper, then the action on~\(Y\) extends uniquely
to an action on the relative Stone--\v{C}ech compactification.
Since the relative Stone--\v{C}ech compactification is a Hausdorff
space with a proper map to the locally compact Hausdorff
space~\(\Gr^0\), it is itself a locally compact Hausdorff space.
Then an abstract nonsense argument shows that the relative
Stone--\v{C}ech compactification of a universal action must be
homeomorphic to the universal action.  This shows that the universal
action lives on a locally compact Hausdorff space that is proper
over~\(\Gr^0\).  As a consequence, this space is compact
if~\(\Gr^0\) is compact.

The main result in this article answers an important, but technical
question in the previous article~\cite{Meyer:Diagrams_models}.
Therefore, we assume that the reader has already
seen~\cite{Meyer:Diagrams_models} and we do not attempt to make this
article self-contained.  In \longref{Section}{sec:preparation}, we
only recall the most crucial results from
\cites{Antunes-Ko-Meyer:Groupoid_correspondences,
  Meyer:Diagrams_models}.  In
\longref{Section}{sec:general_existence}, we prove that any diagram
has a groupoid model --~not necessarily Hausdorff or locally
compact.  In \longref{Section}{sec:relative_SC}, we introduce the
relative Stone--\v{C}ech compactification and prove some properties
that we are going to need.  In \longref{Section}{sec:extend_action},
we prove that an action of an étale groupoid or of a diagram of
proper, locally compact étale groupoid correspondences extends
canonically to the relative Stone--\v{C}ech compactification.  In
\longref{Section}{sec:proper_model}, we use this to prove that the
universal action of such a diagram lives on a space that is
Hausdorff, locally compact, and proper over~\(\Gr^0\).  To conclude,
we discuss two examples.  One of them shows that the groupoid model
may fail to be locally compact if the groupoid correspondences in
the underlying diagram are not proper.

\section{Preparations}
\label{sec:preparation}

In this section, we briefly recall the definition of the bicategory
of groupoid correspondences, diagrams of groupoid correspondences,
their actions on spaces, and the universal action of a diagram.  We
describe actions of diagrams through slices.  More details may be
found in \cites{Antunes-Ko-Meyer:Groupoid_correspondences,
  Meyer:Diagrams_models}.

We describe a topological groupoid~\(\Gr\) by topological spaces
\(\Gr\) and \(\Gr^0\subseteq\Gr\) of arrows and objects with continuous range
and source maps \(\rg,\s\colon \Gr \rightrightarrows \Gr^0\), a
continuous multiplication map
\(\Gr\times_{\s,\Gr^0,\rg} \Gr \to \Gr\), \((g,h)\mapsto g\cdot h\),
such that each object has a unit arrow and each arrow has an inverse
with the usual algebraic properties and the unit map and the
inversion are continuous as well.  We tacitly assume all groupoids
to be \emph{étale}, that is, \(\s\) and~\(\rg\) are local
homeomorphisms.  This implies that each arrow \(g\in\Gr\) has an
open neighbourhood \(\U\subseteq \Gr\) such that \(\s|_\U\)
and~\(\rg|_\U\) are homeomorphisms onto open subsets of~\(\Gr^0\).
Such an open subset is called a \emph{slice}.

\begin{definition}
  An (étale) groupoid~\(\Gr\) is called \emph{locally compact} if
  its object space~\(\Gr^0\) is Hausdorff and locally compact.
\end{definition}

If~\(\Gr\) is a locally compact groupoid, then its arrow
space~\(\Gr\) is locally compact and locally Hausdorff, but it need
not be Hausdorff.  We only know that each slice \(\U\subseteq \Gr\)
is Hausdorff locally compact because it is homeomorphic to an open
subset in~\(\Gr^0\).  As in~\cite{Meyer:Diagrams_models}, we allow
groupoids that are not locally compact.  We need this for the
general existence result for groupoid models.

\begin{definition}[\cite{Meyer:Diagrams_models}*{Definitions 2.7--9}]
  \label{def:Bibundles}
  Let \(\Gr[H]\) and~\(\Gr\) be (étale) groupoids.  An
  \textup{(}étale\textup{)} \emph{groupoid correspondence}
  from~\(\Gr\) to~\(\Gr[H]\), denoted
  \(\Bisp\colon \Gr[H]\leftarrow \Gr\), is a space~\(\Bisp\) with
  commuting actions of \(\Gr[H]\) on the left and~\(\Gr\) on the
  right, such that the right anchor map \(\s\colon \Bisp\to \Gr^0\)
  is a local homeomorphism and the right \(\Gr\)\nb-action is basic.
  A correspondence \(\Bisp\colon \Gr[H]\leftarrow \Gr\) is
  \emph{proper} if the map \(\rg_*\colon \Bisp/\Gr\to \Gr[H]^0\)
  induced by~\(\rg\) is proper.
  Let \(\Gr[H]\) and~\(\Gr\) be locally compact groupoids.  A
  \emph{locally compact groupoid correspondence}
  \(\Bisp\colon \Gr[H]\leftarrow \Gr\) is a groupoid
  correspondence~\(\Bisp\) such that~\(\Bisp/\Gr\) is Hausdorff.
\end{definition}

The ``groupoids'' and ``groupoid correspondences'' as defined
in~\cite{Antunes-Ko-Meyer:Groupoid_correspondences} are the
``locally compact groupoids'' and the ``locally compact groupoid
correspondences'' in the notation in this article.

\begin{definition}[\cite{Antunes-Ko-Meyer:Groupoid_correspondences}*{Definition~7.2}]
  \label{def:correspondence_slices}
  Let \(\Bisp\colon \Gr[H]\leftarrow \Gr\) be a groupoid
  correspondence.  A \emph{slice} of~\(\Bisp\) is an open subset
  \(\U\subseteq \Bisp\) such that both \(\s\colon \Bisp \to \Gr^0\)
  and the orbit space projection \(\Qu\colon \Bisp\prto \Bisp/\Gr\)
  are injective on~\(\U\).  Let \(\Bis(\Bisp)\) be the set of all
  slices of~\(\Bisp\).
\end{definition}

Let \(\Bisp\colon \Gr[H]\leftarrow \Gr\) be a groupoid
correspondence.  Then the slices of~\(\Bisp\) form a basis for the
topology of~\(\Bisp\).

Groupoid correspondences may be composed, and this gives rise to a
bicategory~\(\Grcat\)
(see~\cite{Antunes-Ko-Meyer:Groupoid_correspondences}).  We only
need this structure to talk about bicategory homomorphisms
into~\(\Grcat\).  Such a homomorphism is described more concretely
in~\cite{Meyer:Diagrams_models}:

\begin{proposition}[\cite{Meyer:Diagrams_models}*{Proposition 3.1}]
  \label{pro:diagrams_in_Grcat}
  Let~\(\Cat\) be a category.  A \emph{\(\Cat\)\nb-shaped diagram of
    groupoid correspondences} \(F\colon \Cat\to\Grcat\) is given by
  \begin{enumerate}
    \item groupoids~\(\Gr_x\) for all objects~\(x\) of~\(\Cat\);
    \item correspondences \(\Bisp_g\colon \Gr_x\leftarrow \Gr_y\) for all
    arrows \(g\colon x\leftarrow y\) in~\(\Cat\);
    \item isomorphisms of correspondences \(\mu_{g,h}\colon
    \Bisp_g\Grcomp_{\Gr_y} \Bisp_h\congto \Bisp_{g h}\) for all pairs of
    composable arrows \(g\colon z\leftarrow y\), \(h\colon y\leftarrow
    x\) in~\(\Cat\);
  \end{enumerate}
  such that
  \begin{enumerate}[label=\textup{(\ref*{pro:diagrams_in_Grcat}.\arabic*)},
    leftmargin=*,labelindent=0em]
    \item \label{en:diagrams_in_Grcat_1} \(\Bisp_x\) for an object~\(x\)
    of~\(\Cat\) is the identity correspondence~\(\Gr_x\) on~\(\Gr_x\);
    \item \label{en:diagrams_in_Grcat_2}
    \(\mu_{g,y}\colon \Bisp_g \Grcomp_{\Gr_y} \Gr_y \congto \Bisp_g\)
    and
    \(\mu_{x,g}\colon \Gr_x \Grcomp_{\Gr_x} \Bisp_g \congto \Bisp_g\)
    for an arrow \(g\colon x\leftarrow y\)
    in~\(\Cat\)
    are the canonical isomorphisms;
    \item \label{en:diagrams_in_Grcat_3} for all composable arrows
    \(g_{01}\colon x_0\leftarrow x_1\), \(g_{12}\colon x_1\leftarrow
    x_2\), \(g_{23}\colon x_2\leftarrow x_3\) in~\(\Cat\), the
    following diagram commutes:
    \begin{equation}
      \label{eq:coherence_category-diagram}
      \begin{tikzpicture}[yscale=1.5,xscale=3,baseline=(current bounding
      box.west)]
        \node (m-1-1) at (144:1)
        {\((\Bisp_{g_{01}}\Grcomp_{\Gr_{x_1}} \Bisp_{g_{12}})
          \Grcomp_{\Gr_{x_2}} \Bisp_{g_{23}}\)};
        \node (m-1-1b) at (216:1) {\(\Bisp_{g_{01}}\Grcomp_{\Gr_{x_1}}
          (\Bisp_{g_{12}}\Grcomp_{\Gr_{x_2}} \Bisp_{g_{23}})\)};
        \node (m-1-2) at (72:1)
        {\(\Bisp_{g_{02}}\Grcomp_{\Gr_{x_2}}\Bisp_{g_{23}}\)};
        \node (m-2-1) at (288:1)
        {\(\Bisp_{g_{01}}\Grcomp_{\Gr_{x_1}}\Bisp_{g_{13}}\)};
        \node (m-2-2) at (0:.8) {\(\Bisp_{g_{03}}\)};
        \draw[dar] (m-1-1) -- node[swap] {\(\scriptstyle\cong\)} node
        {\scriptsize\textup{associator}} (m-1-1b);
        \draw[dar] (m-1-1.north) -- node[very near end]
        {\(\scriptstyle\mu_{g_{01},g_{12}}\Grcomp_{\Gr_{x_2}}\id_{\Bisp_{g_{23}}}\)}
         (m-1-2.west);
        \draw[dar] (m-1-1b.south) -- node[swap,very near end]
        {\(\scriptstyle\id_{\Bisp_{g_{01}}}\Grcomp_{\Gr_{x_1}}\mu_{g_{12},g_{23}}\)}
        (m-2-1.west);
        \draw[dar] (m-1-2.south) -- node[inner sep=0pt]
        {\(\scriptstyle\mu_{g_{02},g_{23}}\)} (m-2-2);
        \draw[dar] (m-2-1.north) -- node[swap,inner sep=1pt]
        {\(\scriptstyle\mu_{g_{01},g_{13}}\)} (m-2-2);
      \end{tikzpicture}
    \end{equation}
    here \(g_{02}\defeq g_{01}\circ g_{12}\), \(g_{13}\defeq
    g_{12}\circ g_{23}\), and \(g_{03}\defeq g_{01}\circ g_{12}\circ
    g_{23}\).
  \end{enumerate}
\end{proposition}

\begin{definition}[\cite{Meyer:Diagrams_models}*{Definition 3.8}]
  Let~\(\Cat\) be a category.  A diagram of groupoid correspondences
  \(F\colon \Cat\to\Grcat\) described by the data
  \((\Gr_x,\Bisp_g,\mu_{g,h})\) is \emph{proper} if all the groupoid
  correspondences~\(\Bisp_g\) are proper.  It is \emph{locally
    compact} if all the groupoids~\(\Gr_x\) and the
  correspondences~\(\Bisp_g\) are locally compact.
\end{definition}

\begin{definition}[\cite{Meyer:Diagrams_models}*{Definition 4.5}]
  \label{def:diagram_dynamical_system}
  An \emph{\(F\)\nb-action} on a space~\(Y\) consists of
  \begin{itemize}
  \item a partition \(Y = \bigsqcup_{x\in\Cat^0} Y_x\) into clopen
    subsets;
    \item continuous maps \(\rg\colon Y_x \to \Gr_x^0\);
    \item open, continuous, surjective maps \(\alpha_g\colon \Bisp_g
    \times_{\s,\Gr_x^0,\rg} Y_x \to Y_{x'}\) for arrows \(g\colon
    x'\leftarrow x\) in~\(\Cat\), denoted multiplicatively as
    \(\alpha_g(\gamma,y) = \gamma\cdot y\);
  \end{itemize}
  such that
  \begin{enumerate}[label=\textup{(\ref*{def:diagram_dynamical_system}.\arabic*)},
    leftmargin=*,labelindent=0em]
    \item \label{en:diagram_dynamical_system1}
    \(\rg(\gamma_2\cdot y) = \rg(\gamma_2)\) and \(\gamma_1\cdot
    (\gamma_2\cdot y) = (\gamma_1 \cdot \gamma_2)\cdot y\) for
    composable arrows \(g_1,g_2\) in~\(\Cat\), \(\gamma_1\in
    \Bisp_{g_1}\), \(\gamma_2\in \Bisp_{g_2}\), and \(y\in
    Y_{\s(g_2)}\) with \(\s(\gamma_1) = \rg(\gamma_2)\),
    \(\s(\gamma_2) = \rg(y)\);
    \item \label{en:diagram_dynamical_system2}
    if \(\gamma\cdot y = \gamma'\cdot y'\) for \(\gamma,\gamma'\in
    \Bisp_g\), \(y,y'\in Y_{\s(g)}\), there is \(\eta\in \Gr_{\s(g)}\)
    with \(\gamma' = \gamma\cdot \eta\) and \(y = \eta\cdot y'\);
    equivalently, \(\Qu(\gamma)=\Qu(\gamma')\) for the orbit space
    projection \(\Qu\colon \Bisp_g \to \Bisp_g/\Gr_{\s(g)}\)
    and \(y = \braket{\gamma}{\gamma'} y'\).
  \end{enumerate}
\end{definition}

\begin{definition}[\cite{Meyer:Diagrams_models}*{Definition 4.13}]
  \label{def:universal_F-action}
  An \(F\)\nb-action~\(\Omega\) is \emph{universal} if for any
  \(F\)\nb-action~\(Y\), there is a unique \(F\)\nb-equivariant map
  \(Y\to \Omega\).
\end{definition}

\begin{definition}[\cite{Meyer:Diagrams_models}*{Definition 4.13}]
  \label{def:universal_action}
  A \emph{groupoid model for \(F\)\nb-actions} is an étale
  groupoid~\(\Gr[U]\) with natural bijections between the sets of
  \(\Gr[U]\)\nb-actions and \(F\)\nb-actions on~\(Y\) for all
  spaces~\(Y\).
\end{definition}

It follows from~\cite{Meyer:Diagrams_models}*{Proposition~5.12} that
a diagram has a groupoid model if and only if it has a universal
\(F\)\nb-action.  By definition, an \(F\)\nb-action is universal if
and only if it is terminal in the category of \(F\)\nb-actions.  Our
first goal below will be to prove that any diagram of groupoid
correspondences has a universal \(F\)\nb-action and hence also a
groupoid model.  The universal action and the groupoid model of a
diagram are unique up to canonical isomorphism if they exist (see
\cite{Meyer:Diagrams_models}*{Proposition~4.16}).

A key point in our construction of the universal \(F\)\nb-action is
an alternative description of an \(F\)\nb-action, which uses partial
homeomorphisms associated to slices of the groupoid correspondences
in the diagram.

Let \(\Bisp\colon \Gr[H]\leftarrow \Gr\) be a groupoid
correspondence and let \(\U,\V\subseteq \Bisp\) be slices.  Recall
that \(\braket{x}{y}\) for \(x,y\in\Bisp\) with \(\Qu(x) = \Qu(y)\)
is the unique arrow in~\(\Gr\) with \(x\cdot \braket{x}{y} = y\).
The subset
\[
  \braket{\U}{\V} \defeq
  \setgiven{\braket{x}{y}\in\Gr}{x\in\U,\ y\in\V,\ \Qu(x)=\Qu(y)}
\]
is a slice in the groupoid~\(\Gr\) by
\cite{Antunes-Ko-Meyer:Groupoid_correspondences}*{Lemma~7.7}.  Next,
let \(\Bisp\colon \Gr[H]\leftarrow \Gr\) and
\(\Bisp[Y]\colon \Gr\leftarrow \Gr[K]\) be groupoid correspondences
and let \(\U\subseteq \Bisp\) and \(\V\subseteq \Bisp[Y]\) be
slices.  Then
\[
  \U \cdot \V\defeq
  \setgiven{[x,y] \in \Bisp\Grcomp_{\Gr} \Bisp[Y]}{x\in\U,\
    y\in\V,\ \s(x) = \rg(y)}
\]
is a slice in the composite groupoid correspondence
\(\Bisp\Grcomp_{\Gr} \Bisp[Y]\) by
\cite{Antunes-Ko-Meyer:Groupoid_correspondences}*{Lemma~7.14}.

Let~\(F\) be a diagram of groupoid correspondences.  Let \(\Bis(F)\)
be the set of all slices of the correspondences~\(\Bisp_g\) for all
arrows \(g\in\Cat\), modulo the relation that we identify the empty
slices of \(\Bis(\Bisp_g)\) for all \(g\in \Cat\).  Given composable
arrows \(g,h\in\Cat\) and slices \(\U\subseteq \Bisp_g\),
\(\V\subseteq \Bisp_h\), then \(\U\V \defeq \mu_{g,h}(\U\cdot \V)\)
is a slice in~\(\Bisp_{g h}\).  If \(g,h\) are not composable, then
we let \(\U\V\) be the empty slice~\(\emptyset\).  This
turns~\(\Bis(F)\) into a semigroup with zero element~\(\emptyset\).

\begin{definition}
  Let~\(Y\) be a topological space.  A \emph{partial homeomorphism}
  of~\(Y\) is a homeomorphism between two open subsets of~\(Y\).
  These are composed by the obvious formula: if \(f,g\) are partial
  homeomorphisms of~\(Y\), then \(f g\) is the partial homeomorphism
  of~\(Y\) that is defined on \(y\in Y\) if and only if \(g(y)\) and
  \(f(g(y))\) are defined, and then \((f g)(y) \defeq f(g(y))\).
  If~\(f\) is a partial homeomorphism of~\(Y\), we let~\(f^*\) be
  its ``partial inverse'', defined on the image of~\(f\) by
  \(f^*(f(y)) = y\) for all \(y\) in the domain of~\(f\).
\end{definition}

Let~\(Y\) with the partition \(Y = \bigsqcup_{x\in\Cat^0} Y_x\) be
an \(F\)\nb-action.  Then slices in \(\Bis(F)\) act on~\(Y\) by
partial homeomorphisms.
For an arrow \(g\colon x\leftarrow x'\) in~\(\Cat\), a slice
\(\U\subseteq \Bisp_g\) acts on~\(Y\) by a partial homeomorphism
\[
\vartheta(\U) \colon Y_{x'} \supseteq \rg^{-1}(\s(\U)) \to Y_x,
\]
which maps \(y \in Y_{x'}\) with \(\rg(y) \in \s(\U)\) to
\(\gamma\cdot y\) for the unique \(\gamma\in \U\) with
\(\s(\gamma)=\rg(y)\).  The following lemmas describe
\(F\)\nb-actions and \(F\)\nb-equivariant maps through these partial
homeomorphisms.

\begin{lemma}[\cite{Meyer:Diagrams_models}*{Lemma~5.3}]
  \label{lem:F-action_from_theta}
  Let~\(Y\) be a space and let
  \(\rg\colon Y\to \bigsqcup_{x\in\Cat^0} \Gr_x^0\) and
  \(\vartheta\colon \Bis(F)\to I(Y)\) be maps.  These come from an
  \(F\)\nb-action on~\(Y\) if and only if
  \begin{enumerate}[label=\textup{(\ref*{lem:F-action_from_theta}.\arabic*)},
    leftmargin=*,labelindent=0em]
  \item \label{en:F-action_from_theta1}%
    \(\vartheta(\U \V) = \vartheta(\U)\vartheta(\V)\)
    for all \(\U,\V\in \Bis(F)\);
  \item \label{en:F-action_from_theta2}%
    \(\vartheta(\U_1)^*\vartheta(\U_2) =
    \vartheta(\braket{\U_1}{\U_2})\)
    for all \(g\in\Cat\), \(\U_1,\U_2\in\Bis(\Bisp_g)\);
  \item \label{en:F-action_from_theta3}%
    the images of~\(\vartheta(\U)\) for \(\U\in\Bisp_g\) cover
    \(Y_{\rg(g)} \defeq \rg^{-1}(\Gr_{\rg(g)}^0)\) for each
    \(g\in\Cat\);
  \item \label{en:F-action_from_theta5}%
    \(\rg\circ \vartheta(\U) = \U_\dagger\circ\rg\) as partial maps
    \(Y \to \Gr^0\) for any \(\U\in\Bis(F)\).
  \end{enumerate}
  The corresponding \(F\)\nb-action on~\(Y\) is unique if it exists,
  and it satisfies
  \begin{enumerate}[label=\textup{(\ref*{lem:F-action_from_theta}.\arabic*)},
    leftmargin=*,labelindent=0em,resume]
  \item \label{en:F-action_from_theta4}%
    for \(U\subseteq \Gr_x^0\) open, \(\vartheta(U)\) is the
    identity map on~\(\rg^{-1}(U)\);
  \item \label{en:F-action_from_theta6}%
    for any \(\U\in\Bis(F)\), the domain of \(\vartheta(\U)\) is
    \(\rg^{-1}(\s(\U))\).
  \end{enumerate}
\end{lemma}

\begin{lemma}[\cite{Meyer:Diagrams_models}*{Lemma~5.4}]
  \label{lem:theta_gives_equivariant}
  Let \(Y\) and~\(Y'\) be \(F\)\nb-actions.  A continuous map
  \(\varphi\colon Y\to Y'\) is \(F\)\nb-equivariant if and only if
  \(\rg'\circ \varphi = \rg\) and
  \(\vartheta'(\U)\circ\varphi = \varphi\circ\vartheta(\U)\) for all
  \(\U\in \Bis(F)\).
\end{lemma}

\section{General existence of a groupoid model}
\label{sec:general_existence}

Our next goal is to prove that any diagram of groupoid
correspondences has a groupoid model.  By the results
of~\cite{Meyer:Diagrams_models} mentioned above, it suffices to show
that its category of actions has a terminal object.  Our proof will
use the following criterion for this:

\begin{lemma}
  \label{lem:final_object_exists}
  Let~\(\Cat[D]\) be a cocomplete, locally small category.  Assume
  that there is a set of objects \(\Phi\subseteq \Cat[D]\) such that
  for any object \(x\in\Cat[D]^0\) there is a \(y\in\Phi\) and an
  arrow \(x \to y\).  Then~\(\Cat[D]\) has a terminal object.
\end{lemma}

\begin{proof}
  This is dual to \cite{Riehl:Categories_context}*{Lemma 4.6.5},
  which characterises the existence of an initial object in a
  complete, locally small category.
\end{proof}

\begin{theorem}
  \label{the:groupoid_model_universal_action_exists}
  Any diagram of groupoid correspondences
  \(F\colon \Cat \to \Grcat\) has a universal \(F\)\nb-action and a
  groupoid model.
\end{theorem}

\begin{proof}
  By the discussion above, it suffices to prove that the category of
  \(F\)\nb-actions satisfies the assumptions in
  \longref{Lemma}{lem:final_object_exists}.  We first exhibit the
  set of objects~\(\Phi\).

  Let~\(Y\) be any space with an \(F\)\nb-action.  Equip~\(Y\) with
  the canonical action of the inverse semigroup~\(\IS(F)\).  Call an
  open subset of~\(Y\) \emph{necessary} if it is the domain of some
  element of~\(\IS(F)\).  Let~\(\tau'\) be the topology on~\(Y\)
  that is generated by the necessary open subsets, and let~\(Y'\)
  be~\(Y\) with the topology~\(\tau'\).  Let~\(Y''\) be the quotient
  of~\(Y'\) by the equivalence relation where two points \(y_1,y_2\)
  are identified if
  \[
    \setgiven{U\in\tau'}{y_1 \in U}
    = \setgiven{U\in\tau'}{y_2 \in U}
  \]
  and \(\rg(y_1) = \rg(y_2)\) for the canonical continuous map
  \(\rg\colon Y\to \bigsqcup_{x\in\Cat^0} \Gr_x^0\).  The continuous
  map \(\rg\colon Y\to \Gr^0 \defeq \bigsqcup_{x\in\Cat^0} \Gr_x^0\)
  descends to a map on~\(Y''\), which is continuous because the
  subsets \(\rg^{-1}(U)\) for open subsets \(U\subseteq \Gr_x^0\),
  \(x\in\Cat^0\) are ``necessary'' by \ref{en:F-action_from_theta4}.
  The \(\IS(F)\)\nb-action on~\(Y\) descends to an
  \(\IS(F)\)\nb-action on~\(Y''\) because all the domains of
  elements of \(\IS(F)\) are in~\(\tau'\).  Then
  \longref{Lemma}{lem:F-action_from_theta} implies that the
  \(F\)\nb-action on~\(Y\) descends to an \(F\)\nb-action
  on~\(Y''\).  The quotient map \(Y\prto Y''\) is a continuous
  \(F\)\nb-equivariant map.

  Next, we control the cardinality of the set~\(Y''\).  By
  construction, finite intersections of necessary open subsets form
  a basis of the topology~\(\tau'\).  A point in~\(Y''\) is
  determined by its image in~\(\Gr^0\) and the set of basic open
  subsets that contain it.  This defines an injective map
  from~\(Y''\) to the product of~\(\Gr^0\) and the power
  set~\(\mathcal{P}(\Upsilon)\) for the set~\(\Upsilon\) of finite
  subsets of \(\IS(F)\).  We may use this injective map to transfer
  the \(F\)\nb-action on~\(Y''\) to an isomorphic \(F\)\nb-action on
  a subset of \(\Gr^0\times \mathcal{P}(\Upsilon)\), equipped with
  some topology.  Let~\(\Phi\) be the set of all \(F\)\nb-actions on
  subsets of \(\Gr^0\times \mathcal{P}(\Upsilon)\), equipped with
  some topology.  This is indeed a set, not a class.  The argument
  above shows that any \(F\)\nb-action admits a continuous
  \(F\)\nb-equivariant map to an \(F\)\nb-action in~\(\Phi\), as
  required.

  The category of \(F\)\nb-actions is clearly locally small.  It
  remains to prove that it is cocomplete.  It suffices to prove that
  it has all small coproducts and coequalisers (see
  \cite{Riehl:Categories_context}*{Theorem~3.4.12}).  Coproducts are
  easy: if \((Y_i)_{i\in I}\) is a set of \(F\)\nb-actions, then the
  disjoint union \(\bigsqcup_{i\in I} Y_i\) with the canonical
  topology carries a unique \(F\)\nb-action for which the inclusions
  \(Y_i \to \bigsqcup_{i\in I} Y_i\) are all \(F\)\nb-equivariant,
  and this is a coproduct in the category of \(F\)\nb-actions.  Now
  let \(Y_1\) and~\(Y_2\) be two spaces with \(F\)\nb-actions and
  let \(f,g\colon Y_1 \rightrightarrows Y_2\) be two
  \(F\)\nb-equivariant continuous maps.  Equip~\(Y_2\) with the
  equivalence relation~\(\sim\) that is generated by
  \(f(y)\sim g(y)\) for all \(y\in Y_1\) and let~\(Y\)
  be~\(Y_2/{\sim}\) with the quotient topology.  This is the
  coequaliser of \(f,g\) in the category of topological spaces.  We
  claim that there is a unique \(F\)\nb-action on~\(Y\) so that the
  quotient map is \(F\)\nb-equivariant.  And this \(F\)\nb-action
  turns~\(Y\) into a coequaliser of \(f,g\) in the category of
  \(F\)\nb-actions.  We use \longref{Lemma}{lem:F-action_from_theta}
  to build the \(F\)\nb-action on~\(Y\).  Since \(f,g\) are
  \(F\)\nb-equivariant, the continuous maps
  \(\rg\colon Y_2 \to\Gr^0\) equalises \(f,g\).  Then~\(\rg\)
  descends to a continuous map \(\rg\colon Y\to \Gr^0\).  Let
  \(t\in \IS(F)\).  The domain of~\(t\) is closed under~\(\sim\)
  because \(f,g\) are \(\IS(F)\)\nb-equivariant, and \(y_1\sim y_2\)
  implies \(\vartheta(t)(y_1) \sim \vartheta(t)(y_2)\).  Therefore,
  the image of the domain of~\(\vartheta(t)\) in~\(Y\) is open in
  the quotient topology and \(\vartheta(t)\) descends to a partial
  homeomorphism of~\(Y\).  This defines an action of \(\IS(F)\)
  on~\(Y\).  All conditions in
  \longref{Lemma}{lem:F-action_from_theta} pass from~\(Y_2\)
  to~\(Y\).  We have found an \(F\)\nb-action on~\(Y\).  Any
  continuous map \(h\colon Y_2 \to Z\) with \(h\circ f = h\circ g\)
  descends uniquely to a continuous map \(h^\flat\colon Y\to Z\).
  If~\(h\) is \(F\)\nb-equivariant, then so is~\(h^\flat\) by
  \longref{Lemma}{lem:theta_gives_equivariant}.
  Thus~\(Y\) is a coequaliser of \(f,g\).  This finishes the proof
  that the category of \(F\)\nb-actions is cocomplete.  And then the
  existence of a final object follows.
\end{proof}

\longref{Theorem}{the:groupoid_model_universal_action_exists} has
the merit that it works for any diagram of groupoid correspondences.
For applications to \(\Cst\)\nb-algebras, however, the groupoid
model should be a locally compact groupoid.  Equivalently, the
underlying space~\(\Omega\) of the universal action should be
locally compact and Hausdorff.  \longref{Example}{exa:words} shows
that~\(\Omega\) may fail to be locally compact in rather simple
examples.  In the following sections, we are going to prove
that~\(\Omega\) is locally compact and Hausdorff whenever~\(F\) is a
diagram of proper, locally compact groupoid correspondences.  Like
the proof of
\longref{Theorem}{the:groupoid_model_universal_action_exists}, our
proof of this statement will not be constructive.  The key tool is a
relative form of the Stone--Čech compactification, which we will use
to show that any \(F\)\nb-action maps to an \(F\)\nb-action on a
locally compact Hausdorff space.

\section{The relative Stone--Čech compactification}
\label{sec:relative_SC}

We begin by recalling some well known definitions.

\begin{proposition}[\cite{Bourbaki:Topologie_generale}*{I.10.1,
    I.10.3 Proposition~7}]
  \label{pro:proper_map}
  Let \(X\) and~\(Y\) be topological spaces.  A map
  \(f\colon X\to Y\) is \emph{proper} if and only if
  \(f \times \id_Z\colon X\times Z\to Y\times Z\) is closed for
  every topological space~\(Z\).

  If~\(X\) is Hausdorff and~\(Y\) is Hausdorff, locally compact,
  then \(f\colon X\to Y\) is proper if and only if preimages of
  compact subsets are compact.
\end{proposition}

\begin{definition}[\cite{May-Sigurdsson:Parametrized_homotopy_theory}]
  Let~\(B\) be a topological space.  A \emph{\(B\)\nb-space} is a
  topological space~\(Z\) with a continuous map \(r \colon Z \to B\),
  called anchor map.  It is called \emph{proper} if~\(r\) is a proper
  map.
  Let \((Z_1,r_1)\) and \((Z_2,r_2)\) be two \(B\)\nb-spaces.  A
  \emph{\(B\)\nb-map} is a continuous map \(f \colon Z_1 \to Z_2\) such
  that the following diagram commutes:
  \[
    \begin{tikzcd}[row sep=small]
      Z_1 \ar[rr, "f"] \ar[dr, "r_1"']&&
      Z_2 \ar[dl, "r_2"] \\
      & B
    \end{tikzcd}
  \]
  Let~\(\gps{B}\) be the category of \(B\)\nb-spaces, which has
  \(B\)\nb-spaces as its objects and \(B\)\nb-maps as its morphisms,
  with the usual composition of maps.  Let
  \(\pgps{B} \subseteq \gps{B}\) be the full subcategory of those
  \(B\)\nb-spaces \((Z,r)\) where the space~\(Z\) is Hausdorff and
  the map~\(r\) is proper.
\end{definition}

\begin{remark}
  \label{rem:proper_is_locally_compact}
  If~\(B\) is Hausdorff, locally compact and \((Z,r)\) is a proper
  \(B\)-space, then~\(Z\) is locally compact by
  \longref{Proposition}{pro:proper_map}.  This is how we are going
  to prove that the underlying space of a universal action is
  locally compact.
\end{remark}

For a topological space~\(X\), its Stone--\v{C}ech compactification
is a compact Hausdorff space~\(\beta X\) with a continuous map
\(\iota_X\colon X\to \beta X\), such that any continuous map
from~\(X\) to a compact Hausdorff space factors uniquely
through~\(\iota_X\).  In other words, the Stone--\v{C}ech
compactification~\(\beta\) is left adjoint to the inclusion of the
full subcategory of compact Hausdorff spaces into the category of
all topological spaces.  If~\(B\) is the one-point space, then a
\(B\)\nb-space is just a space, and \(B\)\nb-maps are just
continuous maps.  A proper, Hausdorff \(B\)\nb-space is just a
compact Hausdorff space.  Thus the Stone--\v{C}ech compactification
is a left adjoint for the inclusion \(\pgps{B} \subseteq \gps{B}\)
in the case where~\(B\) is a point.  The \emph{relative}
Stone--\v{C}ech compactification generalises this to all Hausdorff,
locally compact spaces~\(B\).

For a topological space~\(X\), let \(\Contb(X)\) be the
\(\Cst\)\nb-algebra of all bounded, continuous functions \(X\to\C\).
A continuous map \(f\colon X\to Y\) induces a \Star{}homomorphism
\(f^*\colon \Contb(Y) \to \Contb(X)\), \(h\mapsto h\circ f\).
If~\(X\) is Hausdorff, locally compact, then we let
\(\Cont_0(X) \subseteq \Contb(X)\) be the ideal of all continuous
functions \(X\to\C\) that vanish at~\(\infty\).  If \(X\) and~\(Y\)
are Hausdorff, locally compact spaces and \(f\colon X\to Y\) is a continuous
map, then the restriction of \(f^*\colon \Contb(X)\to \Contb(Y)\) to
\(\Cont_0(X)\) is \emph{nondegenerate}, that is,
\[
  f^*(\Cont_0(X)) \cdot \Cont_0(Y) = \Cont_0(Y).
\]
Conversely, any nondegenerate \Star{}homomorphism is of this form
for a unique continuous map~\(f\).  The range of~\(f^*\) is
contained in \(\Cont_0(Y)\) if and only if~\(f\) is proper.

\begin{definition}
  Let~\(B\) be a locally compact Hausdorff space and let \((X,r)\) be a
  \(B\)\nb-space.  The \emph{relative Stone--\v{C}ech
    compactification} \(\rscc{B} X\) of~\(X\) over~\(B\) is defined as
  the spectrum of the \(C^*\)\nb-subalgebra
  \[
    H_X \defeq \Contb(X) \cdot r^*(\Cont_0(B)) \subseteq \Contb(X).
  \]
\end{definition}

We show that the relative Stone--\v{C}ech compactification is indeed
the reflector (left adjoint) of the inclusion \(\pgps{B} \hookrightarrow \gps{B}\).

In the following, we let~\(B\) be a locally compact Hausdorff space,
\((X,r)\) an object in \(\gps{B}\) and \((X',r')\) an object in
\(\pgps{B}\).  Then~\(X'\) is Hausdorff by the definition of
\(\pgps{B}\) and locally compact by
\longref{Remark}{rem:proper_is_locally_compact}.

The inclusion \(i^* \colon H_X \hookrightarrow \Contb(X)\) is a
\Star{}homomorphism.  For each \(x \in X\), denote by
\(\mathrm{ev}_x\) the evaluation map at~\(x\).  Then
\(\mathrm{ev}_x\circ i^* \colon H_X \to \C\) is a character
on~\(H_X\).  It is nonzero on~\(H_X\) because
\(\mathrm{ev}_x\circ i^*(1\cdot r^*(h))\neq0\) if \(h\in\Cont_0(B)\)
satisfies \(h(r(x))\neq 0\).  Thus \(\mathrm{ev}_x\circ i^*\) is a
point in the spectrum~\(\rscc{B} X\) of~\(H_X\).  This defines a map
\(i \colon X \to \rscc{B} X\).  The map~\(i\) is continuous because
\(h\circ i\) is continuous for all
\(h\in H_X = \Cont_0(\rscc{B} X)\).

\begin{lemma}
  \label{uniquedual}
  Let \(f, g \colon X \rightrightarrows X'\).  If \(f \neq g\), then
  \(f^* \neq g^*\colon \Cont_0(X') \to \Contb(X)\).
\end{lemma}

\begin{proof}
  By assumption, there is \(x \in X\) with \(f(x) \neq g(x)\) in~\(X'\).
  Since~\(X'\) is Hausdorff and locally compact, we may separate
  \(f(x)\) and~\(g(x)\) by relatively compact, open neighbourhoods
  \(U_f\) and~\(U_g\).  Urysohn's Lemma gives a continuous function
  \(h\colon \overline{U_f} \to [0,1]\) with \(h(f(x))=1\) and
  \(h|_{\partial U_f} = 0\).  Extend~\(h\) by~\(0\) to a
  function~\(\tilde{h}\) on~\(X'\).  This belongs to \(\Cont_0(X')\)
  because \(h|_{\partial U_f} = 0\) and~\(\overline{U_f}\) is
  compact, and \(\tilde{h}(g(x))=0\).  Thus
  \(f^*(\tilde{h}) \neq g^*(\tilde{h})\).
\end{proof}

\begin{lemma}
  \label{dense}
  Let~\(S\) be a subset of a locally compact Hausdorff space~\(X'\).  If
  the restriction map from \(\Cont_0(X')\) to \(\Contb(S)\) is
  injective, then~\(S\) is dense in~\(X'\).
\end{lemma}

\begin{proof}
  We prove the contrapositive statement.  Suppose that~\(S\) is not
  dense in~\(X'\).  Then \(\overline{S} \neq X'\).  As in the proof of
  \longref{Lemma}{uniquedual}, there is a nonzero continuous
  function \(h\in \Cont_0(X' \backslash \overline{S})\).
  Extending~\(h\) by zero gives a nonzero function in
  \(\Cont_0(X')\) that vanishes on~\(S\).
\end{proof}

\begin{lemma}
  \label{lem:X_dense_in_beta}
  The image of~\(X\) in \(\rscc{B} X\) is dense.
\end{lemma}

\begin{proof}
  \longref{Lemma}{dense} shows this because
  \(i^*\colon \Cont_0(\rscc{B} X) \cong H_X \to \Contb(X)\) is
  injective.
\end{proof}

\begin{proposition}
  \label{betaXtoY}
  Let \(f \colon X \to X'\) be a morphism in \(\gps{B}\).  Assume~\(X'\)
  to be a Hausdorff proper \(B\)\nb-space.  Then there is a
  unique continuous map \(f' \colon \rscc{B} X \to X'\) such
  that the following diagram commutes:
  \[
    \begin{tikzcd}[sep=small]
      X \ar[rr, "f"] \ar[dr,  "i"']&& X'
      \\
      & \rscc{B}  X  \ar[ur,dashed,  "\exists !  f'"']
    \end{tikzcd}
  \]
  The map~\(f'\) is automatically proper.
\end{proposition}

\begin{proof}
  Let \(f^* \colon \Cont_0(X') \to \Contb(X)\) be the dual map
  of~\(f\) and let \(i^* \colon H_X \hookrightarrow \Contb(X)\) be
  the inclusion map.  Since~\(r'\) is proper, it induces a
  nondegenerate \Star{}homomorphism
  \((r')^* \colon \Cont_0(B) \to \Cont_0(X')\).  We use this to show
  that \(f^*(\Cont_0(X')) \subseteq H_X\):
  \begin{multline*}
    f^*(\Cont_0(X'))
    = f^*({r'}^*(\Cont_0(B)) \cdot \Cont_0(X'))
    \\= f^*({r'}^*(\Cont_0(B))) \cdot f^*(\Cont_0(X'))
    = r^*(\Cont_0(B)) \cdot f^*(\Cont_0(X'))
    \subseteq H_X.
  \end{multline*}
  Let~\((f^*)'\) be~\(f^*\) viewed as a \Star{}homomorphism
  \(\Cont_0(X') \to H_X\).  We claim that~\((f^*)'\) is
  nondegenerate.  The proof uses that a \Star{}homomorphism is
  nondegenerate if and only if it maps an approximate unit again to
  an approximate unit; this well known result goes back at least to
  \cite{Rieffel:Induced_Banach}*{Proposition~3.4}.  Let
  \((e_i)_{i \in I}\) be an approximate unit in \(\Cont_0(B)\).
  Then \({r'}^*(e_i)\) is an approximate unit in \(\Cont_0(X')\).
  Now
  \((f^*)'({r'}^*(e_i)) = r^*(e_i) = i^* ({\rscc{B}} r)^* (e_i)\).
  For any \(\varphi_1 \in \Contb(X)\) and
  \(\varphi_2 \in \Cont_0(B)\),
  \(\norm{\varphi_1 \cdot r^*(\varphi_2)r^*(e_i) - \varphi_1 \cdot
    r^*(\varphi_2)} \le \norm{\varphi_1} \norm{r^*(\varphi_2 e_i) -
    r^*(\varphi_2)} \to 0\), as \(r^*\) is continuous.  Hence
  \(r^*(e_i)\) is an approximate unit in~\(H_X\).  We let
  \(f' \colon \rscc{B} X \to X'\) be the dual of~\((f^*)'\).  This
  is a proper continuous map.  Since~\(X'\) is Hausdorff, two
  continuous maps to~\(X'\) that are equal on a dense subset are
  equal everywhere.  Therefore, \(f'\) is unique by
  \longref{Lemma}{lem:X_dense_in_beta}.
\end{proof}

\begin{corollary}
  \label{2inj}
  The anchor map \(r \colon X \to B\) extends uniquely to a proper
  continuous map \({\rscc{B}} r \colon {\rscc{B}} X \to B\), such that
  the following diagram commutes:
  \[
    \begin{tikzcd}[row sep=small]
      X \ar[rr,  "i"] \ar[dr, "r"']&&
      \rscc{B} X
      \ar[dl, dashed,
      "\exists! \rscc{B} r"] \\
      & B
    \end{tikzcd}
  \]
\end{corollary}

\begin{proof}
  Since the identity map \(B \to B\) is proper, \(B\) is an object in
  \(\pgps{B}\).  Now apply \longref{Proposition}{betaXtoY} in the case
  where \(X' = B\) and \(f = r \colon X \to B\).
\end{proof}

\begin{proposition}
  In the above setting, the following diagram commutes:
  \[
    \begin{tikzcd}[row sep=small]
      {\rscc{B}}  X \ar[rr, "f'"] \ar[dr, "{\rscc{B}}  r"']&&
      X'\ar[dl,  "r'"]
      \\
      & B
    \end{tikzcd}
  \]
\end{proposition}

\begin{proof}
  We get
  \(i^* \circ (f^*)' \circ (r')^* = i^* \circ ({\rscc{B}} r)^*\) by
  construction.  Since~\(i^*\) is a monomorphism, this implies
  \((f^*)' \circ (r')^* = (\rscc{B} r)^*\).
\end{proof}

\begin{theorem}
  \label{the:rscc_reflector}
  \(\rscc{B}\) is a reflector or, equivalently, it is left adjoint to
  the inclusion functor
  \(I \colon \pgps{B} \hookrightarrow \gps{B}\).
\end{theorem}

\begin{proof}
  The propositions above tell us that~\(\rscc{B}\) is left adjoint
  to~\(I\).
\end{proof}

\begin{lemma}
  \label{propercorr}
  Let~\(X\) be a topological space, and let \(Y\) and~\(Z\) be
  locally compact Hausdorff spaces.  Let \(f_1 \colon X \to Y\) be
  continuous and let \(f_2 \colon Y \to Z\) be proper and
  continuous.  Then
  \(\rscc{Y} (X,f_1) \cong \rscc{Z} (X,f_2\circ f_1)\).
\end{lemma}

\begin{proof}
  It suffices to show that
  \(\Contb(X) \cdot {f_1}^*(\Cont_0(Y)) = \Contb(X) \cdot
  (f_2f_1)^*(\Cont_0(Z))\).  Since~\(f_2\) is proper,
  \(f_2^* \colon \Cont_0(Z) \to \Cont_0(Y)\) is nondegenerate.  In
  particular, \(f_1^*(f_2^*(\Cont_0(Z)) \subseteq f_1^*(\Cont_0(Y))\),
  giving the inclusion~``\(\supseteq\)''.  Since
  \(f_2^*(\Cont_0(Z)) \cdot \Cont_0(Y) = \Cont_0(Y)\), we compute
  \[
    f_1^*(\Cont_0(Y)) = f_1^*(f_2^*(\Cont_0(Z) \cdot \Cont_0(Y)))
    = (f_2f_1)^*(\Cont_0(Z)) \cdot f_1^*(\Cont_0(Y))
  \]
  and then
  \[
    \Contb(X) \cdot f_1^*(\Cont_0(Y)) = \Contb(X) \cdot
    (f_2f_1)^*(\Cont_0(Z)) \cdot f_1^*(\Cont_0(Y)) \subseteq \Contb(X)
    \cdot (f_2f_1)^*(\Cont_0(Z)).\qedhere
  \]
\end{proof}

\begin{lemma}
  \label{sq}
  In a commuting diagram of topological spaces and continuous maps
  \[
    \begin{tikzcd}
      X_1 \ar[rr,  "f" ] \ar[d, "r_1"']&& X_2
      \ar[d,  "r_2"] \\
      B_1 \ar[rr,  "f_0"] && B_2,
    \end{tikzcd}
  \]
  assume \(B_1\) and~\(B_2\) to be locally compact Hausdorff
  and~\(f_0\) to be proper.  Then there is a unique continuous map
  \(\tilde{f} \colon \rscc{B_1} X_1 \to \rscc{B_2} X_2\) that makes
  the following diagram commute:
  \[
    \begin{tikzcd}
      X_1 \ar[rr,  "f" ] \ar[d, "i_1"']&& X_2
      \ar[d,  "i_2"] \\
      \rscc{B_1} X_1 \ar[rr, dotted, "\tilde{f}"] \ar[d, "\rscc{B_1}
      r_1"']&& \rscc{B_2} X_2
      \ar[d,  "\rscc{B_2} r_2"] \\
      B_1 \ar[rr,  "f_0"] && B_2
    \end{tikzcd}
  \]
\end{lemma}

\begin{proof}
  Since~\(f_0\) is proper, \(B_1\) is an object in the category of
  Hausdorff proper \(B_2\)\nb-spaces.  Then
  \longref{Theorem}{the:rscc_reflector} implies
  \(\rscc{B_2} B_1 \cong B_1\) and gives a commuting diagram
  \[
    \begin{tikzcd}
      X_1 \ar[rr,  "f" ] \ar[d, "i_1^2"']&& X_2
      \ar[d, "i_2"] \\
      \rscc{B_2} X_1 \ar[rr,  "\rscc{B_2} f" ] \ar[d, "\rscc{B_2} r_1"']&&
      \rscc{B_2} X_2
      \ar[d,  "\rscc{B_2} r_2"] \\
      B_1 \ar[rr,  "f_0"] && B_2.
    \end{tikzcd}
  \]
  \longref{Lemma}{propercorr} gives
  \(\rscc{B_1} X_1 = \rscc{B_2} X_1\), and this turns the diagram
  above into what we need.  The map \(\tilde{f} =\rscc{B_2} f\) is
  unique because~\(\rscc{B_2} X_2\) is Hausdorff and the image
  of~\(X_1\) in~\(\rscc{B_2} X_1\) is dense by
  \longref{Lemma}{lem:X_dense_in_beta}.
\end{proof}

\begin{lemma}
  \label{equalityextension}
  Given a commuting diagram of continuous maps
  \[
    \begin{tikzcd}
      X_1 \ar[rr, bend left, "f"] \ar[r,  "h"'] \ar[d, "\rg_1"']&
      X_2 \ar[r, "g"' ] \ar[d,  "\rg_2"] &
      X_3 \ar[d,  "\rg_3"] \\
      B_1 \ar[r, "h_0"] \ar[rr, bend right, "f_0"']
      & B_2 \ar[r, "g_0"] & B_3
    \end{tikzcd}
  \]
  with locally compact Hausdorff spaces~\(B_j\) and proper \(h_0\)
  and~\(g_0\), the maps constructed in \longref{Lemma}{sq} satisfy
  \(\tilde{f} = \tilde{g} \circ \tilde{h}\).
\end{lemma}

\begin{proof}
  The map~\(\tilde{g} \circ \tilde{h}\) also has the properties that
  uniquely characterise~\(\tilde{f}\).
\end{proof}

\begin{lemma}
  \label{extiso}
  If the map~\(f\) in \longref{Lemma}{sq} is a homeomorphism, then so
  is~\(\tilde{f}\).
\end{lemma}

\begin{proof}
  Apply \longref{Lemma}{equalityextension} to the compositions
  \(f\circ f^{-1}\) and \(f^{-1}\circ f\).
\end{proof}

The following results will be used in the next section to extend an
action of a diagram to the relative Stone--\v{C}ech
compactification.

\begin{lemma}
  \label{lem:rsc_disjoint_union}
  Let~\(I\) be a set, let~\(B_i\) for \(i\in I\) be locally compact
  Hausdorff spaces, and let \(r_i\colon Y_i\to B_i\) be topological spaces
  over~\(B_i\).  Let \(B = \bigsqcup_{i\in I} B_i\) and
  \(Y = \bigsqcup_{i\in I} Y_i\) with the induced map
  \(r\colon Y\to B\).  Then
  \(\rscc{B} Y \cong \bigsqcup_{i\in I} \rscc{B_i} Y_i\).
\end{lemma}

\begin{proof}
  The map that takes the family \((Y_i,r_i)\) of spaces over~\(B_i\)
  to \((Y,r)\) as a space over~\(B\) is an equivalence of categories
  from the product of categories \(\prod_{i\in I} \gps{B_i}\) to the
  category \(\gps{B}\).  A space over~\(B\) is Hausdorff and proper
  if and only if its pieces over~\(B_i\) are Hausdorff and proper
  for all \(i\in I\).  That is, the isomorphism of categories above
  identifies the subcategory \(\pgps{B}\) of Hausdorff and proper
  \(B\)\nb-spaces with the product of the subcategories
  \(\pgps{B_i}\).  The product of the reflectors
  \(\rscc{B_i}\colon \gps{B_i} \to \pgps{B_i}\) is a reflector
  \(\prod_{i\in I} \gps{B_i} \to \prod_{i\in I} \pgps{B_i}\).  Under
  the equivalence above, this becomes the reflector~\(\rscc{B}\).
  Both reflectors must be naturally isomorphic.
\end{proof}

\begin{lemma}
  \label{fonpreimg}
  Let~$\Gr$ be a locally compact groupoid, let~$V$ be an open subset
  of~$\Gr^0$, and let $(Z,k\colon Z\to\Gr^0)$ be a locally compact
  Hausdorff space over~$\Gr^0$.  Then
  $\Cont_0(k^{-1}(V)) \cong k^*(\Cont_0(V)) \cdot \Cont_0(Z)$.
\end{lemma}

\begin{proof}
  Let $J \defeq k^*(\Cont_0(V)) \cdot \Cont_0(Z)$.  This is
  an ideal in $\Cont_0(Z)$.  So its spectrum~\(\hat{J}\) is an open
  subset of~$Z$.  Namely, it consists of those \(z\in Z\) for which
  there is \(f\in J\) with \(f(z)\neq0\).  There is always
  \(h\in \Cont_0(Z)\) with \(h(z) \neq0\).  Therefore,
  \(z\in\hat{J}\) if and only if there is \(g\in \Cont_0(V)\) with
  \(k^*(g)(z)\neq0\).  Since \(k^*(g)(z) = g(k(z))\), such a~\(g\)
  exists if and only if \(k(z)\in V\).  Thus
  \(\hat{J} = k^{-1}(V)\).
\end{proof}

\begin{lemma}
  \label{lem:ident}
  Let~\(B\) be a locally compact Hausdorff space and let
  \(V\subseteq B\) be an open subset.  Let \((Y,\rg_Y)\) be
  a space over~\(B\) and let
  \((\rscc{B} Y,\rg_{\rscc{B} Y}) \in \gps{B}\) be its relative
  Stone--\v{C}ech compactification.  Then
  \(\rg_Y^{-1}(V) \subseteq Y\) is a space over~\(V\), so that
  \(\rscc{V} (\rg_Y^{-1} (V))\) is defined, and
  \(\rscc{V} (\rg_Y^{-1} (V)) \cong (\rg_{\rscc{B} Y})^{-1} (V)
  \subseteq\rscc{B} Y\).
\end{lemma}

\begin{proof}
  We proceed in terms of their \(\Cst\)\nb-algebras.  By definition
  of the relative Stone--\v{C}ech compactification,
  $\rscc{V} (r_Y^{-1} (V))$ is the spectrum of the commutative
  \(\Cst\)\nb-algebra
  $\Contb(r_Y ^{-1}(V)) \cdot r_Y ^* (\Cont_0(V))$.  By
  \longref{Lemma}{fonpreimg}, $(\rsc r_Y)^{-1} (V)$ corresponds to
  $\Cont_0(\rsc Y) \cdot r_Y ^* (\Cont_0(V))$.  Since
  $\Cont_0(\Gr^0) \cdot \Cont_0(V) = \Cont_0(V)$, we compute
  \[
    \Cont_0(\rsc Y) \cdot r_Y^*(\Cont_0(V))
    = \Contb(Y) \cdot r_Y^*(\Cont_0(\Gr^0)) \cdot r_Y^*(\Cont_0(V))
    = \Contb(Y) \cdot r_Y^*(\Cont_0(V)).
  \]
  Therefore, it suffices to show that
  $\Contb(r_Y^{-1}(V)) \cdot r_Y^*(\Cont_0(V)) \cong \Contb(Y)
  \cdot r_Y^*(\Cont_0(V))$.

  Bounded functions on~$Y$ restrict to bounded functions
  on~$r_Y^{-1} (V)$, and this restriction map is injective on the
  subalgebra $\Contb(Y) \cdot r_Y ^*(\Cont_0(V))$ because functions
  in this subalgebra vanish outside $r_Y^{-1}(V)$.  Therefore, there
  is an inclusion
  \[
    \varrho \colon
    \Contb(Y) \cdot r_Y ^* (\Cont_0(V)) \hookrightarrow
    \Contb(r_Y^{-1}(V)) \cdot r_Y ^* (\Cont_0(V)).
  \]
  We must prove that it is surjective.  Any element of
  \(\Contb(r_Y^{-1}(V)) \cdot r_Y ^* (\Cont_0(V))\) is of the form
  \(f\cdot h\) with $f \in \Contb(r_Y ^{-1}(V))$ and
  $h \in r_Y ^* (\Cont_0(V))$.  The Cohen--Hewitt Factorisation
  Theorem gives $h_1, h_2 \in r_Y ^* (\Cont_0(V))$ with
  $h = h_1 \cdot h_2$.  Let $\varphi$ be the
  extension of $f \cdot h_1 \colon r_Y ^{-1}(V) \to \C$ by zero.  We
  are going to show that $\varphi$ is continuous
  on~$Y$.  Since
  \(\varrho(\varphi\cdot h_2) = f\cdot h\), it
  follows that~\(\varrho\) is surjective.

  It remains to prove that \(\varphi\) is
  continuous.  The only points where this is unclear are the
  boundary points of \(r_Y^{-1}(V)\).  Let \((y_n)_{n\in N}\) be a
  net that converges towards such a boundary point.  We claim
  that \(\varphi(y_n)\) converges to~\(0\).  This
  proves the claim.  If \(y_n \notin r_Y^{-1}(V)\), then
  \(\varphi(y_n)=0\) by construction.  So it is no
  loss of generality to assume \(y_n \in r_Y^{-1}(V)\) for all
  \(n\in N\).  Then \(r_Y(y_n)\) is a net in~\(V\) that
  converges towards~\(\infty\).  Therefore,
  \(\lim h_1'(r_Y(y_n)) = 0\) for all \(h_1' \in \Cont_0(V)\).  This
  implies \(\lim h_1(y_n) = 0\).  Since~\(f\) is bounded, this
  implies \(\lim \varphi(y_n)=0\).
\end{proof}

\section{Extending actions to the relative Stone--Čech compactification}
\label{sec:extend_action}

The aim of this section is to extend an action of a diagram on a
topological space~\(Y\) with the anchor map
\(\rg_Y \colon Y \to \Gr^0\) to~\(\rsc Y\).  Actions of étale
groupoids are a special case of such diagram actions, and this
special case is a bit easier.  Therefore, we first treat only
actions of groupoids.  Since our aim is to generalise to diagram
actions, we do not complete the proof in this case, however.  We
only prove a more technical result about the action of slices of the
groupoid.

Let~\(\Gr\) be a locally compact étale groupoid acting on a
topological space~\(Y\) with the anchor map
\(\rg_Y \colon Y \to \Gr^0\).  The action of~\(\Gr\) on~\(Y\) may be
encoded as in \longref{Lemma}{lem:F-action_from_theta} by the anchor
map \(\rg_Y\colon Y\to \Gr^0\) and partial homeomorphisms
\(\vartheta_Y(\U)\) of~\(Y\) for all slices \(\U\in \Bis(\Gr)\)
on~\(Y\), subject to some conditions.  In fact, in this case the
conditions simplify quite a bit, but we do not go into this here.
The anchor map~\(\rg_Y\) extends to a continuous map
\(\rsc \rg_Y\colon \rsc Y \to \Gr^0\) by construction.  The
following lemma describes the canonical extension of the partial
homeomorphisms~\(\vartheta_Y(\U)\):

\begin{proposition}
  \label{extendgpaction}
  Let \(\U \in \Bis(\Gr)\).  The partial
  homeomorphism~\(\vartheta_Y (\U)\) of~\(Y\) extends uniquely to a
  partial homeomorphism
  \[
    \vartheta_{\rsc Y}(\U)\colon
    (\rsc \rg_Y)^{-1} (\s(\U)) \congto (\rsc \rg_Y)^{-1} (\rg(\U)).
  \]
  Here ``extends'' means that
  \(\vartheta_{\rsc Y}(\U) \circ i_Y = i_Y \circ \vartheta_Y(\U)\)
  for the canonical map \(i_Y\colon Y \to \rsc Y\).
\end{proposition}

\begin{proof}
  The anchor map \(\rg_Y\colon Y\to\Gr^0\) is \(\Gr\)\nb-equivariant
  when we let~\(\Gr\) act on~\(\Gr^0\) in the usual way.  The
  slice~\(\U\) acts both on~\(Y\) and on~\(\Gr^0\), and the latter
  action is the composite homeomorphism
  \(\rg|_\U \circ (\s|_U)^{-1}\colon \s(\U) \congto \U \congto
  \rg(\U)\).  The naturality of the construction of~\(\vartheta\)
  shows that the following diagram commutes:
  \[
    \begin{tikzcd}
      \rg_Y^{-1} (\s(\U)) \ar[r,  "   \vartheta_Y (\U)", "   \cong"']
      \ar[d, "\rg_Y"']&
      \rg_Y^{-1} (\rg(\U))
      \ar[d, "\rg_Y"] \\
      \s(\U) \ar[r,  "\cong"', "\vartheta_{\Gr^0}(\U)"] & \rg(\U)
    \end{tikzcd}
  \]

  Now \longref{Lemma}{sq} with \(B_1 = \s(\U)\) and
  \(B_2 = \rg(\U)\) gives a map
  \[
    \widetilde{\vartheta_Y (\U)} \colon \rscc{\s(\U)} (\rg_Y^{-1}
    (\s(\U))) \to \rscc{\rg(\U)} (\rg_Y^{-1} (\rg(\U))).
  \]
  It is a homeomorphism by \longref{Lemma}{extiso}.
  \longref{Lemma}{lem:ident} identifies the domain and codomain of
  \(\widetilde{\vartheta_Y (\U)}\) with
  \((\rsc \rg_Y)^{-1} (\s(\U))\) and \((\rsc \rg_Y)^{-1} (\rg(\U))\)
  as spaces over \(\s(\U)\) and \(\rg(\U)\), respectively.  So we
  get a partial homeomorphism \(\vartheta_{\rscc{\Gr^0} Y} (\U)\) of
  \(\rscc{\Gr^0} Y\) that makes the following diagram commute:
  \begin{equation}
    \label{eq:vartheta_beta_U_diagram}
    \begin{tikzcd}[column sep=large]
      \rg_Y^{-1} (\s(\U)) \ar[r,  "   \vartheta_Y (\U)", "   \cong"']
      \ar[d, "i_Y"']&
      \rg_Y^{-1} (\rg(\U))
      \ar[d, "i_Y"] & Y\\
      \rg_{\rscc{\Gr^0} Y}^{-1} (\s(\U))
      \ar[r,  "\vartheta_{\rscc{\Gr^0} Y} (\U)", "\cong"']
      \ar[d, "\rg_{\rscc{\Gr^0}Y}"']&
      \rg_{\rscc{\Gr^0} Y}^{-1} (\rg(\U))
      \ar[d, "\rg_{\rscc{\Gr^0}Y}"] & \rsc Y\\
      \s(\U) \ar[r,  "\cong"', "\vartheta_{\Gr^0}(\U)"] & \rg(\U) & \Gr^0
    \end{tikzcd}
  \end{equation}
  The argument also shows
  \(\vartheta_{\rsc Y}(\U) \circ i_Y = i_Y \circ \vartheta_Y(\U)\)
  and that the image of \(\rg_Y^{-1}(\s(\U))\) in
  \((\rsc \rg_Y)^{-1} (\rg(\U))\) is dense.  Since the space
  \(\rsc Y\) is Hausdorff, this implies that the top
  square~\eqref{eq:vartheta_beta_U_diagram} determines the extension
  \(\vartheta_{\rscc{\Gr^0} Y} (\U)\) uniquely.
\end{proof}

To show that the \(\Gr\)\nb-action on~\(Y\) extends uniquely to a
\(\Gr\)\nb-action on~\(\rscc{\Gr^0} Y\), it would remain to prove
that the partial homeomorphisms \(\vartheta_{\rscc{\Gr^0} Y} (\U)\)
for slices \(\U\in\Bis(\Gr)\) satisfy the conditions in
\longref{Lemma}{lem:F-action_from_theta}.  We will prove this in the
more general case of diagram actions.

Before we continue to this more general case, we rewrite the
diagram~\eqref{eq:vartheta_beta_U_diagram} in a way useful for the
generalisation to \(F\)\nb-actions below.  We claim
that~\eqref{eq:vartheta_beta_U_diagram} commutes if and only if the
following diagram commutes, where the dashed arrows are partial
homeomorphisms and the usual arrows are globally defined continuous
maps:
\begin{equation}
  \label{eq:vartheta_beta_U_diagram_partial}
  \begin{tikzcd}[column sep=huge]
    Y \ar[r, dashed,  "\vartheta_Y(\U)"] \ar[d, "i_Y"']&
    Y \ar[r, dashed,  "\vartheta_Y(\U)^*"] \ar[d, "i_Y"']&
    Y \ar[d, "i_Y"]\\
    \rscc{\Gr^0} Y
    \ar[r, dashed,  "\vartheta_{\rscc{\Gr^0} Y} (\U)"]
    \ar[d, "\rg_{\rscc{\Gr^0}Y}"']&
    \rscc{\Gr^0} Y
    \ar[r, dashed,  "\vartheta_{\rscc{\Gr^0} Y} (\U)^*"]
    \ar[d, "\rg_{\rscc{\Gr^0}Y}"] &
    \rscc{\Gr^0} Y
    \ar[d, "\rg_{\rscc{\Gr^0}Y}"'] \\
    \Gr^0 \ar[r, dashed, "\vartheta_{\Gr^0}(\U)"] &
    \Gr^0 \ar[r, dashed, "\vartheta_{\Gr^0}(\U)^*"] &
    \Gr^0
  \end{tikzcd}
\end{equation}
A diagram of partial maps commutes if and only if any two parallel
partial maps in the diagram are equal, and this includes an equality
of their domains.  The domain of
\(\rg_{\rsc Y}\circ \vartheta_{\rsc Y}(\U)\) is equal to the domain
of \(\vartheta_{\Gr^0}(\U)\), whereas the domain of
\(\vartheta_{\Gr^0}(\U) \circ \rg_{\rsc Y}\) is
\(\rg_{\rsc Y}^{-1}(\s(\U))\) because \(\vartheta_{\Gr^0}(\U)\) has
domain~\(\s(\U)\).  Thus the bottom left square implies that
\(\vartheta_{\rsc Y}(\U)\) has the domain
\(\rg_{\rsc Y}^{-1}(\s(\U))\).  Similarly, the bottom right square
implies that \(\vartheta_{\rsc Y}(\U)^*\) has the domain
\(\rg_{\rsc Y}^{-1}(\rg(\U))\).  Equivalently,
\(\vartheta_{\rsc Y}(\U)\) has the image
\(\rg_{\rsc Y}^{-1}(\rg(\U))\).  In the top row, the domain and
image of \(\vartheta_Y(\U)\) must be \(\rg_Y^{-1}(\s(\U))\) and
\(\rg_Y^{-1}(\rg(\U))\) for the diagram to commute.  In addition,
the diagram commutes as a diagram of ordinary maps when we replace
each entry by the domain of the partial maps that start there.  This
gives exactly~\eqref{eq:vartheta_beta_U_diagram}.  So the
diagram~\eqref{eq:vartheta_beta_U_diagram_partial} encodes both the
commutativity of~\eqref{eq:vartheta_beta_U_diagram} and the domains
and images of the partial maps in that diagram.

Now let~\(\Cat\) be a category and let \((\Gr_x, \Bisp_g, \mu_{g,h})\)
describe a \(\Cat\)\nb-shaped diagram
\(F \colon \Cat \to \Grcat_{\lc,\prop}\).  That is, each~\(\Gr_x\)
for \(x\in\Cat^0\) is a locally compact, étale groupoid,
each~\(\Bisp_g\) for \(g\in \Cat(x,x')\) is a proper, locally
compact, étale groupoid correspondence
\(\Bisp_g \colon \Gr_{x'} \leftarrow \Gr_x\), and each~\(\mu_{g,h}\)
for \(g,h\in \Cat\) with \(\s(g) = \rg(h)\) is a homeomorphism
\(\mu_{g,h}\colon \Bisp_g \Grcomp_{\Gr_{\s(g)}} \Bisp_h \congto
\Bisp_{g h}\), subject to the conditions in
\longref{Proposition}{pro:diagrams_in_Grcat}.  Let~\(Y\) be a
topological space with an action of~\(F\).  The action contains a
disjoint union decomposition \(Y = \bigsqcup_{x\in\Cat^0} Y_x\) and
continuous maps \(\rg_x\colon Y_x \to \Gr_x^0\), which we assemble
into a single continuous map \(\rg\colon Y\to \Gr^0\) with
\(\Gr^0 \defeq \bigsqcup_{x\in\Cat^0} \Gr_x^0\).  This makes~\(Y\) a
space over~\(Y\) and allows us to define the Stone--\v{C}ech
compactification \(\rscc{\Gr^0}{Y}\) of~\(Y\) relative to~\(\Gr^0\).
We are going to extend the action of~\(F\) on~\(Y\) to an action
on~\(\rsc Y\).

The key is the description of \(F\)\nb-actions in
\longref{Lemma}{lem:F-action_from_theta}.  The space~\(\rsc Y\) comes
with a canonical map \(\rg_{\rsc Y}\colon \rsc Y \to \Gr^0\), which is
one piece of data assumed in
\longref{Lemma}{lem:F-action_from_theta}.  We are going to construct
partial homeomorphisms~\(\vartheta_{\rsc Y}(\U)\) for all
\(\U\in\Bis(F)\) and then check the conditions in
\longref{Lemma}{lem:F-action_from_theta}.  Before we start, we
notice that, by \longref{Lemma}{lem:rsc_disjoint_union},
\[
  \rsc Y = \bigsqcup_{x\in\Cat^0} \rscc{\Gr_x^0} Y_x.
\]

\begin{lemma}
  \label{lem:bisections_act_on_rsc}
  Let \(\U\in\Bis(\Bisp_g)\) for some \(x,x'\in \Cat^0\) and
  \(g\in\Cat(x,x')\).  There is a unique partial homeomorphism
  \(\vartheta_{\rsc Y}(\U)\) from \(\rscc{\Gr_x^0} Y_x\) to
  \(\rscc{\Gr_{x'}^0} Y_{x'}\) that makes the following diagram
  commute:
  \begin{equation}
    \label{eq:vartheta_beta_U_diagram_partial_corr}
    \begin{tikzcd}[column sep=huge]
      Y_x \ar[r, dashed,  "\vartheta_Y(\U)"] \ar[d, "i_{Y_x}"']&
      Y_{x'} \ar[r, dashed,  "\vartheta_Y(\U)^*"] \ar[d, "i_{Y_{x'}}"']&
      Y_x \ar[d, "i_{Y_x}"]\\
      \rscc{\Gr^0_x} Y_x
      \ar[r, dashed,  "\vartheta_{\rscc{\Gr^0} Y} (\U)"]
      \ar[d, "\rg_{\rscc{\Gr^0_x}Y_x}"']&
      \rscc{\Gr^0_{x'}} Y_{x'}
      \ar[r, dashed,  "\vartheta_{\rscc{\Gr^0} Y} (\U)^*"]
      \ar[d, "\pi"']
      \ar[dd, bend left, near end, "\rg_{\rscc{\Gr^0_{x'}}Y_{x'}}"] &
      \rscc{\Gr^0_x} Y_x
      \ar[d, "\rg_{\rscc{\Gr^0_x}Y_x}"] \\
      \Gr_x^0 \ar[r, dashed, "\U_*"] \ar[rd, dotted, "\U_\dagger"'] &
      \Bisp_g/\Gr_x \ar[r, dashed, "(\U_*)^*"] \ar[d, "\rg_*"'] &
      \Gr^0_x \\
      & \Gr_{x'}^0
    \end{tikzcd}
  \end{equation}
  Here continuous maps are drawn as usual arrows, partial
  homeomorphisms as dashed arrows, and one partial map is drawn as a
  dotted arrow.
\end{lemma}

\begin{proof}
  We first recall how the arrows \(\U_*\), \(\rg_*\)
  and~\(\U_\dagger\)
  in~\eqref{eq:vartheta_beta_U_diagram_partial_corr} are defined and
  check that the triangle they form commutes.  Since~\(\U\) is a
  slice, \(\s|_{\U}\colon \U\congto \s(\U)\subseteq \Gr_x^0\) and
  \(\Qu|_{\U}\colon \U \congto \Qu(\U) \subseteq \Bisp_g/\Gr_x\) are
  homeomorphisms onto open subsets.  This yields the partial
  homeomomorphism
  \(\U_*\defeq\Qu|_{\U} \circ (\s|_{\U})^{-1}\colon \s(\U) \congto
  \Qu(\U)\).  The map \(\rg_*\colon \Bisp_g/\Gr_x \to \Gr_{x'}^0\)
  in~\eqref{eq:vartheta_beta_U_diagram_partial_corr} is induced by
  the anchor map \(\rg\colon \Bisp_g \to \Gr_{x'}^0\).  By
  definition,
  \(\U_\dagger \defeq \rg_*\circ \U_*\colon \s(\U) \to \Qu(\U) \to
  \rg(\U) \subseteq \Gr_{x'}^0\).

  The vertical maps in the first and third column of
  diagram~\eqref{eq:vartheta_beta_U_diagram_partial_corr} and the
  maps \(i_{Y_{x'}}\) and \(\rg_{\rscc{\Gr^0_{x'}}Y_{x'}}\) in the
  second column are part of the construction of the relative
  Stone--\v{C}ech compactification.  Next we construct a map
  \(\pi\colon \rscc{\Gr_{x'}^0} Y_{x'} \to \Bisp_g/\Gr_x\) with
  \(\rg_*\circ \pi = \rg_{\rscc{\Gr^0_{x'}}Y_{x'}}\).  There is a
  canonical map
  \(\pi_Y\colon Y_{x'}\congto \Bisp_g \Grcomp Y_x \to
  \Bisp_g/\Gr_x\) that maps \(\gamma\cdot y\) for
  \(\gamma\in\Bisp_g\), \(y\in Y_x\) with \(\s(\gamma) = \rg_Y(y)\)
  to the right \(\Gr_x\)\nb-orbit of~\(\gamma\); this is well
  defined because \(\gamma\cdot y = \gamma_2\cdot y_2\) implies
  \(\gamma_2 = \gamma\cdot \eta\) and \(y_2= \eta^{-1}\cdot y\) for some
  \(\eta\in \Gr\) with \(\rg(\eta) = \s(\gamma)\).  We compute
  \(\rg_* \circ \pi_Y= \rg_Y\colon Y_{x'} \to \Gr_{x'}^0\) because
  \(\rg_*\circ \pi_Y(\gamma\cdot y) = \rg(\gamma) =
  \rg_Y(\gamma\cdot y)\) for all \(\gamma\in\Bisp_g\), \(y\in Y_x\)
  with \(\s(\gamma) = \rg_Y(y)\).  So~\(\pi_Y\) is a map
  over~\(\Gr_{x'}^0\).  By assumption, the space \(\Bisp_g/\Gr_x\)
  is Hausdorff and the map
  \(\rg_*\colon \Bisp_g/\Gr_x \to \Gr_{x'}^0\) is proper.  By
  \longref{Theorem}{the:rscc_reflector}, \(\pi_Y\) factors uniquely
  through a proper, continuous map
  \(\pi\colon \rscc{\Gr^0} Y \to \Bisp_g/\Gr_x\)
  over~\(\Gr_{x'}^0\).  That this is a map over~\(\Gr_{x'}^0\) means
  that
  \(\rg_*\circ \pi = \rg_{\rscc{\Gr^0} Y}\colon \rscc{\Gr_{x'}^0}
  Y_{x'} \to \Gr_{x'}^0\).

  Next we recall the construction of the partial
  homeomorphism~\(\vartheta_Y(\U)\) from~\(Y_x\) to~\(Y_{x'}\) and
  prove that
  \begin{equation}
    \label{eq:slice_pi_equation}
    (\U_*)^* \circ \pi_Y=\rg_Y \circ \vartheta_Y(\U)^*.
  \end{equation}
  By construction,
  \(\vartheta_Y(\U)\) has the domain \(\rg_Y^{-1}(\s(\U))\) and is
  defined by \(\vartheta_Y(\U)(y) \defeq \gamma\cdot y\) if
  \(y\in\rg_Y^{-1}(\s(\U))\) and \(\gamma\in\U\) is the unique element
  with \(\s(\gamma) = \rg_Y(y)\).  As a consequence,
  \(\pi_Y (\vartheta_Y(\U)(y)) = \Qu(\gamma) = \U_*(\s(\gamma)) =
  \U_*(\rg_Y(y))\).  Since the partial maps
  \(\pi_Y \circ \vartheta_Y(\U)\) and \(\U_* \circ \rg_Y\) both have
  the domain \(\rg_Y^{-1}(\s(\U))\), we conclude that
  \(\pi_Y \circ \vartheta_Y(\U) = \U_* \circ \rg_Y\) as partial maps
  from~\(Y_x\) to \(\Bisp_g/\Gr_x\).  We claim that the partial maps
  \((\U_*)^* \circ \pi_Y\) and \(\rg_Y \circ \vartheta_Y(\U)^*\)
  from~\(Y_x\) to \(\Gr_x^0\) are equal as well.  The first of them
  has domain \(\pi_Y^{-1}(\Qu(\U))\) because the image of~\(\U_*\)
  is \(\Qu(\U)\), and the domain of the second one is the image of
  \(\vartheta_Y(\U)\).  Therefore, we must show that the image of
  \(\vartheta_Y(\U)\) is equal to
  \(\pi_Y^{-1}(\Qu(\U)) \subseteq Y_{x'}\).

  It is clear that~\(\pi_Y\) maps the image of~\(\vartheta_Y(\U)\)
  into~\(\Qu(\U)\).  Conversely, let
  \(z \in \pi_Y^{-1}(\Qu(\U)) \subseteq Y_{x'}\).  There are
  \(\gamma\in\Bisp_g\), \(y\in Y_{x'}\) with \(\s(\gamma) = \rg(y)\) and
  \(z = \gamma\cdot y\).  Then \(\pi_Y(z) \defeq \Qu(\gamma)\), and this
  belongs to~\(\Qu(\U)\) by assumption.  Therefore, there is a
  unique \(\eta\in \Gr\) with \(\s(\gamma) = \rg(\eta)\) and
  \(\gamma\cdot \eta \in \U\).  Then
  \(z=(\gamma \eta)\cdot (\eta^{-1} y)=\vartheta_Y(\U)(\eta^{-1}y)\).  So~\(z\)
  belongs to the image of~\(\vartheta_Y(\U)\).  In addition, we get
  \[
    \rg_Y(\vartheta_Y(\U)^*(z))
    = \rg_Y(\eta^{-1} y)
    = \rg(\eta^{-1})
    = \s(\eta)
    = \s(\gamma\cdot \eta)
    = (\U_*)^*\Qu(\gamma)
    = (\U_*)^*\pi_Y(z).
  \]
  This finishes the proof of~\eqref{eq:slice_pi_equation}.

  As in the proof of \longref{Proposition}{extendgpaction}, we now
  apply \longref{Lemma}{sq} and \longref{Lemma}{extiso} with
  \(B_1 = \s(\U)\) and \(B_2 = \Qu(\U)\) to get a unique
  homeomorphism
  \[
    \widetilde{\vartheta_Y (\U)} \colon
    \rscc{\s(\U)} (\rg_Y^{-1}(\s(\U)))
    \congto \rscc{\Qu(\U)} (\pi_Y^{-1} (\Qu(\U)))
  \]
  with \(i_Y \vartheta_Y (\U) = \widetilde{\vartheta_Y (\U)} i_Y\)
  on \(\rg_Y^{-1}(\s(\U)) \subseteq Y_x\).  Then
  \longref{Lemma}{lem:ident} identifies the domain and codomain of
  \(\widetilde{\vartheta_Y (\U)}\):
  \begin{align*}
    \rscc{\s(\U)} (\rg_Y^{-1}(\s(\U)))
    &\cong (\rsc \rg_Y)^{-1} (\s(\U)) \subseteq \rsc Y_x,\\
    \rscc{\Qu(\U)} (\pi_Y^{-1} (\Qu(\U)))
    &\cong \pi^{-1}(\Qu(\U)) \subseteq \rscc{\Bisp_g/\Gr_x} Y_{x'}.
  \end{align*}
  \longref{Lemma}{lem:ident} identifies
  \(\rscc{\Bisp_g/\Gr_x} Y_{x'}\) with the Stone--\v{C}ech
  compactification of~\(Y_{x'}\) relative to~\(\Gr_{x'}^0\) because
  \(\rg_*\colon \Bisp_g/\Gr_x \to \Gr_{x'}^0\) is proper.
  Composing~\(\widetilde{\vartheta_Y (\U)}\) with these
  homeomorphisms gives a partial homeomorphism
  \(\vartheta_{\rscc{\Gr^0} Y} (\U)\) of \(\rscc{\Gr^0} Y\) that
  makes the diagram~\eqref{eq:vartheta_beta_U_diagram_partial_corr}
  commute.  It is unique because the target space is Hausdorff
  and~\(i_Y\) maps \(\rg_Y^{-1}(\s(\U))\) to a dense subset of its
  domain, where the top left square
  in~\eqref{eq:vartheta_beta_U_diagram_partial_corr} determines
  \(\vartheta_{\rscc{\Gr^0} Y} (\U)\).
\end{proof}

\begin{theorem}
  \label{the:unique_extension_F-action}
  Let \(F\colon \Cat\to\Grcat_{\lc,\prop}\) be a diagram of proper,
  locally compact groupoid correspondences.  Let~\(Y\) be a
  topological space with an \(F\)\nb-action.  There is a unique
  \(F\)\nb-action on \(\rsc Y\) such that the canonical map
  \(i_Y\colon Y\to \rsc Y\) is \(F\)\nb-equivariant.
\end{theorem}

\begin{proof}
  The Stone--\v{C}ech compactification relative to
  \(\Gr^0 \defeq \bigsqcup_{x\in\Cat^0} \Gr_x^0\) is well defined
  because \(\bigsqcup_{x \in \Cat^0} \Gr_x^0\) is locally compact
  and Hausdorff.  There is a canonical map
  \(\rsc r\colon \rsc Y \to \Gr^0\).  It is the unique map with
  \(\rsc r\circ i_Y = r\colon Y\to\Gr^0\).  Hence this is the only
  choice for an anchor map if we want~\(i\) to be
  \(F\)\nb-equivariant.  \longref{Lemma}{lem:bisections_act_on_rsc}
  provides partial homeomorphisms \(\vartheta_{\rsc Y}(\U)\) of
  \(\rsc Y\) for all slices \(\U\in\Bis(F)\).  We claim that these
  satisfy the conditions in
  \longref{Lemma}{lem:F-action_from_theta}.

  We first check~\ref{en:F-action_from_theta1}.  Let
  \(\U\in\Bis(\Bisp_g)\), \(\V\in\Bis(\Bisp_h)\) for
  \(g\in\Cat(x,x')\), \(h\in \Cat(x'',x)\) for
  \(x,x',x''\in\Cat^0\).  The diagram
  in~\eqref{eq:vartheta_beta_U_diagram_partial_corr} describes the
  domain and the codomain of the maps \(\vartheta_{\rsc Y}(\U)\) as
  the preimages of \(\s(\U)\) and~\(\Qu(\U)\), respectively.  The
  domain of \(\vartheta_{\rsc Y}(\U)\vartheta_{\rsc Y}(\V)\) is the
  set of \(y\in Y_{x''}\) with \(\rg_{\rsc Y}(y) \in \s(\V)\) and
  \(\vartheta_{\rsc Y}(\V)(y) \in \rg_{\rsc Y}^{-1}(\s(\U))\).
  Since
  \(\rg_{\rsc Y}\circ \vartheta_{\rsc Y}(\V) = \V_*\circ \rg_{\rsc
    Y}\), the second condition on~\(y\) is equivalent to
  \(\V_*(\rg_{\rsc Y}(y)) \in \s(\U)\).  As a consequence,
  \(\vartheta_{\rsc Y}(\U)\vartheta_{\rsc Y}(\V)\) and
  \(\vartheta_{\rsc Y}(\U \V)\) have the same domain.  The diagram
  in~\eqref{eq:vartheta_beta_U_diagram_partial_corr} also implies
  \(\vartheta_{\rsc Y}(\U)\vartheta_{\rsc Y}(\V) i_Y =
  \vartheta_{\rsc Y}(\U \V) i_Y\).  Since the target space
  \(\rsc Y\) of \(\vartheta_{\rsc Y}(\U)\vartheta_{\rsc Y}(\V)\) and
  \(\vartheta_{\rsc Y}(\U \V)\) is Hausdorff and
  \(i_Y(\rg_Y^{-1}(\s(\U\V)))\) is dense in the domain
  \(\rg_{\rsc Y}^{-1}(\s(\U\V))\) of our two partial maps, we get
  \(\vartheta_{\rsc Y}(\U)\vartheta_{\rsc Y}(\V) = \vartheta_{\rsc
    Y}(\U \V)\).

  The proof of~\ref{en:F-action_from_theta2} is similar, using also
  the right half of~\eqref{eq:vartheta_beta_U_diagram_partial_corr}.
  To prove \longref{condition}{en:F-action_from_theta3}, we use that
  the range of~\(\vartheta_{\rsc Y}(\U)\) is \(\pi^{-1}(\Qu(\U))\).
  These open subsets for slices~\(\U\) of~\(\Bisp_g\)
  cover~\(\rsc Y_{x'}\) because the open subsets
  \(\Qu(\U) \subseteq \Bisp_g/\Gr_x\) for slices~\(\U\)
  cover~\(\Bisp_g/\Gr_x\).  Finally,
  \longref{condition}{en:F-action_from_theta5} is already contained
  in~\eqref{eq:vartheta_beta_U_diagram_partial_corr}.

  Now \longref{Lemma}{lem:F-action_from_theta} shows that the
  map~\(\rg_{\rsc Y}\) and the partial homeomorphisms
  \(\vartheta_{\rsc Y}(\U)\) give a unique \(F\)\nb-action
  on~\(\rsc Y\).  By
  \longref{Lemma}{lem:theta_gives_equivariant}, the top
  part of~\eqref{eq:vartheta_beta_U_diagram_partial_corr} says that
  the map~\(i_Y\) is \(F\)\nb-equivariant.  In addition, since this
  determines the partial homeomorphisms \(\vartheta_{\rsc Y}(\U)\)
  uniquely, the \(F\)\nb-action on~\(\rsc Y\) is unique as asserted.
\end{proof}

\section{Locally compact groupoid models for proper diagrams}
\label{sec:proper_model}

In this subsection, we prove the main result of this article,
namely, that the universal action of a diagram of proper, locally
compact groupoid correspondences takes place on a Hausdorff proper
\(\Gr^0\)-space~\(\Omega\).  Since~\(\Gr^0\) is Hausdorff, locally compact, it
follows that~\(\Omega\) is Hausdorff, locally compact.  The key point is the
following proposition:

\begin{proposition}
  \label{pro:F-action_reflector}
  Let \(F\colon \Cat\to\Grcat_{\lc,\prop}\) be a diagram of proper,
  locally compact groupoid correspondences.  The full subcategory of
  \(F\)\nb-actions on Hausdorff proper \(\Gr^0\)\nb-spaces is a
  reflective subcategory of the category of all \(F\)\nb-actions.  The
  left adjoint to the inclusion maps an \(F\)\nb-action on a
  space~\(Y\) to the induced \(F\)\nb-action on the Stone--\v{C}ech
  compactification of~\(Y\) relative to
  \(\Gr^0 \defeq \bigsqcup_{x \in \Cat^0} \Gr_x^0\).
\end{proposition}

\begin{proof}
  \longref{Theorem}{the:rscc_reflector} says that the full
  subcategory of Hausdorff proper \(\Gr^0\)\nb-spaces is a reflective
  subcategory of the category of all \(\Gr^0\)\nb-spaces, with the
  relative Stone--\v{C}ech compactification~\(\rsc\) as the left
  adjoint functor of the inclusion.

  Let \(Y\) and~\(Y'\) be topological spaces with an action of~\(F\)
  and let \(\varphi\colon Y\to Y'\) be an \(F\)\nb-equivariant map.
  Assume that~\(Y'\) is Hausdorff and that its anchor map
  \(\rg'\colon Y' \to \Gr^0\) is proper.  By
  \longref{Theorem}{the:unique_extension_F-action}, there is a
  unique \(F\)\nb-action on the relative Stone--\v{C}ech
  compactification \(\rsc Y\) that makes the inclusion map
  \(i_Y\colon Y\to \rsc Y\) \(F\)\nb-equivariant.  By
  \longref{Proposition}{betaXtoY}, there is a unique
  \(\Gr^0\)\nb-map \(\tilde{\varphi}\colon \rsc Y \to Y'\) with
  \(\tilde{\varphi} i_Y = \varphi\).  Any \(F\)\nb-equivariant map
  is also a \(\Gr^0\)-map by
  \longref{Lemma}{lem:theta_gives_equivariant}.  Therefore,
  \(\tilde{\varphi}\) is the only map \(\rsc Y\to Y'\) with
  \(\tilde{\varphi} i_Y = \varphi\) that has a chance to be
  \(F\)\nb-equivariant.  To complete the proof, we must show that
  \(\tilde{\varphi}\) is \(F\)\nb-equivariant.  We describe an
  \(F\)\nb-action on a space~\(Y\) as in
  \longref{Lemma}{lem:F-action_from_theta} through a continuous map
  \(\rg_Y\colon Y\to \Gr^0\) and partial homeomorphisms
  \(\vartheta_Y(\U)\) for all slices \(\U\in\Bis(F)\), subject to
  the conditions
  \ref{en:F-action_from_theta1}--\ref{en:F-action_from_theta5}.  By
  \longref{Lemma}{lem:theta_gives_equivariant}, it remains to prove
  that the partial maps \(\vartheta_{Y'}(\U)\circ \tilde{\varphi}\)
  and \(\tilde{\varphi} \circ \vartheta_{\rsc Y}(\U)\) agree for any
  slice \(\U\in\Bis(F)\).  We pick~\(\U\).  Then~\(\U\) is a slice
  in \(\Bisp_g\) for some \(x,x'\in\Cat^0\) and \(g \in\Cat(x,x')\).

  First, we check that our two partial maps have the same domain.
  Since~\(\tilde{\varphi}\) is a globally defined map, the domain of
  \(\tilde{\varphi} \circ \vartheta_{\rsc Y}(\U)\) is the domain of
  \(\vartheta_{\rsc Y}(\U)\) and the domain
  of~\(\vartheta_{Y'}(\U)\circ \tilde{\varphi}\) is the
  \(\tilde{\varphi}\)\nb-preimage of the domain
  of~\(\vartheta_{Y'}(\U)\).  The domains of
  \(\vartheta_{\rsc Y}(\U)\) and~\(\vartheta_{Y'}(\U)\) are
  \(\rg_{\rsc Y_x}^{-1}(\s(\U))\) and \(\rg_{Y'_x}^{-1}(\s(\U))\),
  respectively.  Since~\(\tilde{\varphi}\) is a \(\Gr^0\)\nb-map,
  the domain of~\(\vartheta_{Y'}(\U)\circ \tilde{\varphi}\) is also
  equal to the \(\rg_{\rsc Y_x}^{-1}(\s(\U))\).  This proves the
  claim that both partial maps have the same domain.

  Since \(i_Y\) and~\(\varphi\) are \(F\)\nb-equivariant, we know
  that
  \(\vartheta_{\rsc Y}(\U) \circ i_Y = i_Y \circ \vartheta_Y(\U)\)
  and
  \(\vartheta_{Y'}(\U) \circ \varphi = \varphi \circ
  \vartheta_Y(\U)\).  Together with \(\tilde{\varphi} \circ i_Y =
  \varphi\), this implies
  \[
    (\vartheta_{Y'}(\U) \circ \tilde{\varphi}) \circ i_Y
    = \vartheta_{Y'}(\U) \circ \varphi
    = \varphi \circ \vartheta_Y(\U)
    = \tilde{\varphi} \circ i_Y  \circ \vartheta_Y(\U)
    = (\tilde{\varphi} \circ \vartheta_{\rsc Y}(\U)) \circ i_Y.
  \]
  These partial maps have domain \(\rg_{Y_x}^{-1}(\s(\U))\).  The
  \(i_Y\)\nb-image of this is dense in
  \(\rg_{\rsc Y_x}^{-1}(\s(\U))\) because of
  \longref{Lemma}{lem:ident} and
  \longref{Lemma}{lem:X_dense_in_beta}.  Since the target~\(Y'\) of
  \(\vartheta_{Y'}(\U)\circ \tilde{\varphi}\) and
  \(\tilde{\varphi} \circ \vartheta_{\rsc Y}(\U)\) is Hausdorff and
  these maps agree on a dense subset, we get
  \(\vartheta_{Y'}(\U)\circ \tilde{\varphi}= \tilde{\varphi} \circ
  \vartheta_{\rsc Y}(\U)\) as needed.
\end{proof}

\begin{proposition}[\cite{Riehl:Categories_context}*{Corollary~5.6.6}]
  \label{riehl}
  The inclusion of a reflective full subcategory
  \(\Cat[D] \hookrightarrow \Cat\) creates all limits that~\(\Cat\)
  admits.  As a consequence, if a diagram in~\(\Cat[D]\) has a limit
  in~\(\Cat\), then it also has a limit in~\(\Cat[D]\), which is
  isomorphic to the limit in~\(\Cat\).
\end{proposition}

\begin{theorem}
  \label{the:proper_Hausdorff_Omega}
  Let \(F\colon \Cat\to\Grcat_{\lc,\prop}\) be a diagram of proper,
  locally compact groupoid correspondences.  Then the universal
  \(F\)\nb-action takes place on a space~\(\Omega\) that is
  Hausdorff, locally compact and proper over
  \(\Gr^0 \defeq \bigsqcup_{x\in\Cat^0} \Gr_x^0\).  The groupoid
  model of~\(F\) is a locally compact groupoid.
\end{theorem}

\begin{proof}
  We give two proofs.  First, a universal \(F\)\nb-action is the
  same as a terminal object in the category of \(F\)\nb-actions, and
  this is an example of a limit, namely, of the empty diagram.
  \longref{Theorem}{the:groupoid_model_universal_action_exists} says
  that a terminal object exists in the category of all
  \(F\)\nb-actions.  \longref{Proposition}{pro:F-action_reflector}
  and \longref{Proposition}{riehl} imply that this terminal object
  is isomorphic to an object in the subcategory of Hausdorff proper
  \(\Gr\)\nb-spaces.  Actually, our subcategory is closed under
  isomorphism, and so the terminal object belongs to it.  By
  \longref{Remark}{rem:proper_is_locally_compact}, this implies that
  its underlying space is locally compact.

  The second proof is more explicit.  Let~\(\Omega\) be the
  universal \(F\)\nb-action.  The relative Stone--\v{C}ech
  compactification comes with a canonical \(F\)\nb-equivariant map
  \(\iota \colon \Omega \hookrightarrow \rsc \Omega\); here we use
  the canonical \(F\)\nb-action on~\(\rsc \Omega\).
  Since~\(\Omega\) is universal, there is a canonical map
  \(\rsc \Omega \to \Omega\) as well.  The composite map
  \(\Omega \to \rsc \Omega \to \Omega\) is the identity map
  because~\(\Omega\) is terminal.  The composite map
  \(\rsc \Omega \to \Omega \to \rsc \Omega\) and the identity map
  have the same composite with~\(\iota\).  Since the range
  of~\(\iota\) is dense by \longref{Lemma}{lem:X_dense_in_beta}
  and~\(\rsc \Omega\) is Hausdorff, it follows that the composite
  map \(\rsc \Omega \to \Omega \to \rsc \Omega\) is equal to the
  identity map as well.  So \(\Omega \cong \rsc \Omega\), and this
  means that~\(\Omega\) is Hausdorff and proper over~\(\Gr^0\).
\end{proof}

\begin{corollary}
  \label{cor:Omega_compact}
  Let \(F\colon \Cat\to\Grcat_{\lc,\prop}\) be a diagram of proper,
  locally compact groupoid correspondences.  Assume that~\(\Cat^0\)
  is finite and that each object space~\(\Gr_x^0\) in the diagram is
  compact.  Then the universal \(F\)\nb-action takes place on a
  compact Hausdorff space.  The groupoid model of~\(F\) is a locally
  compact groupoid with compact object space.
\end{corollary}

\begin{proof}
  Our extra assumptions compared to
  \longref{Theorem}{the:proper_Hausdorff_Omega} say that~\(\Gr^0\)
  is compact.  Then Hausdorff spaces that are proper over~\(\Gr^0\)
  are compact.
\end{proof}

\begin{example}
  \label{exa:nm-dynamical_system}
  The \((m,n)\)-dynamical systems of Ara, Exel and
  Katsura~\cite{Ara-Exel-Katsura:Dynamical_systems} are described in
  \cite{Meyer:Diagrams_models}*{Section~4.4} as actions of a certain
  diagram of proper groupoid correspondences.  The diagram is an
  equaliser diagram of the form \(\Gr \rightrightarrows \Gr\),
  where~\(\Gr\) is the one-arrow one-object groupoid.  A proper
  groupoid correspondence \(\Gr \to \Gr\) is just a finite set, and
  it is determined up to isomorphism by its cardinality.  We get
  \((m,n)\)-dynamical systems when we pick the two sets to have
  cardinality \(m\) and~\(n\), respectively.
  \longref{Corollary}{cor:Omega_compact} applies to this diagram and
  shows that its universal action takes place on a compact Hausdorff
  space.  Ara, Exel and Katsura describe in
  \cite{Ara-Exel-Katsura:Dynamical_systems}*{Theorem~3.8} an
  \((m,n)\)-dynamical system that is universal among
  \((m,n)\)-dynamical systems on compact Hausdorff spaces.
  \longref{Corollary}{cor:Omega_compact} shows that it remains
  universal if we allow \((m,n)\)-dynamical systems on arbitrary
  topological spaces.
\end{example}

\begin{example}
  \label{exa:words}
  Let \(\Cat = (\N, +)\) be the category with a single object and
  morphisms the nonnegative integers.  A diagram
  \(F\colon \Cat\to\Grcat\) is determined by a single groupoid
  correspondence \(\Bisp\colon \Gr\leftarrow\Gr\) for an étale
  groupoid~\(\Gr\) (see \cite{Meyer:Diagrams_models}*{Section~3.4}).
  Let~\(\Gr\) be the trivial groupoid with one arrow and one object.
  Then~\(\Bisp\) is just a discrete set because the source map
  \(\Bisp \to \Gr^0\) is a local homeomorphism.  The groupoid model
  of the resulting diagram is a special case of the self-similar
  groups treated in~\cite{Meyer:Diagrams_models}*{Section~9.2}, in
  the case when the group is trivial.  It is shown there that the
  universal action takes place on the space
  \(\Omega \defeq \prod_{n\in\N} \Bisp\).  If~\(\Bisp\) is finite,
  then~\(\Omega\) is compact by Tychonoff’s Theorem.  In contrast,
  if~\(\Bisp\) is infinite, then~\(\Omega\) is not even locally
  compact.  This example shows that we need a diagram of proper
  correspondences for the groupoid model to be locally compact.
\end{example}

\end{document}